\pdfoutput=1
\documentclass[12pt]{article}
\usepackage{amsthm}
\usepackage{amsmath}
\usepackage{amssymb}
\usepackage{amsthm}
\usepackage{amscd}
\usepackage[margin=1in]{geometry}
\usepackage{comment}
\usepackage{pdfpages}
\usepackage{tikz-cd}
\usepackage{breqn}
\usepackage{kbordermatrix}
\usepackage{mathtools}
\usepackage{enumerate}
\usepackage{graphicx}
\usepackage{hyperref}

\newtheorem{thm}{Theorem}[section]
\newtheorem{lem}[thm]{Lemma}

\newtheorem{conjecture}[thm]{Conjecture}
\newtheorem{thm*}{Theorem}

\theoremstyle{definition}

\newtheorem{example}[thm]{Example}

\numberwithin{equation}{section}

\newcommand{\C}{\mathbb{C}}

\newcommand{\FF}{\mathbb{F}}

\newcommand{\mf}{\mathfrak}

\def\ker{\mbox{\rm ker}}

\newcommand{\mb}{\mathbb}

\DeclareMathOperator{\rank}{rank}
\DeclareMathOperator{\Gr}{Gr}
\DeclareMathOperator{\ad}{ad}
\DeclareMathOperator{\Sym}{Sym}
\newcommand{\Fgen}{\mb{F}^{\mathrm{gen}}}
\newcommand{\Rgen}{\widehat{R}_\mathrm{gen}}
\newcommand{\Plucker}{\mathrm{Pl\ddot{u}cker}}

\def\ker{\mbox{\rm ker}}
\def\deg{\mbox{\rm deg}}

\begin{document}

\title{Structure theorems for\\
Gorenstein ideals of codimension four\\
with small number of generators}

\author{Tymoteusz Chmiel \thanks{ Jagiellonian University, Instytut Matematyki, Krak\'{o}w \textbf{Email address:} tymoteusz.chmiel@uj.edu.pl 
} 
\and Lorenzo Guerrieri \thanks{ Jagiellonian University, Instytut Matematyki, Krak\'{o}w \textbf{Email address:} lorenzo.guerrieri@uj.edu.pl 
} 
\and Xianglong Ni
\thanks{University of Notre Dame, Department of Mathematics \textbf{Email address:} xni3@nd.edu } 
\and Jerzy Weyman
\thanks{Jagiellonian University, Instytut Matematyki, Krak\'{o}w \textbf{Email address:} jerzy.weyman@uj.edu.pl }} 

\maketitle

\begin{abstract}
\noindent
In this article we study minimal free resolutions of Gorenstein ideals of codimension four, using methods coming from representation theory. We introduce families of higher structure maps associated with such resolution, defined similarly to the codimension three case. As our main application, we prove that every Gorenstein ideal of codimension four minimally generated by six elements is a hyperplane section of a Gorenstein ideal of codimension three, strengthening a result by Herzog-Miller and Vasconcelos-Villarreal. We state analogous conjectural results for ideals minimally generated by seven and eight elements. 
\end{abstract}

\section*{Introduction}

Gorenstein ideals have applications in singularity theory and linkage theory, and occur in the study of canonical curves, $K3$ surfaces and more general Calabi-Yau varieties. The structure theory of Gorenstein ideals of codimension $4$ has been pursued since the paper \cite{Buchsbaum-Eisenbud_codim-3} of Buchsbaum and Eisenbud on the codimension $3$ case.

Kunz showed in \cite{kunz} that an almost complete intersection cannot be Gorenstein, therefore the smallest possible case to consider was the case of Gorenstein ideals of codimension 4 and deviation $2$ (the deviation of an ideal is classically defined as the difference of minimal number of generators and codimension).
Herzog and Miller \cite{herzog-miller} proved that under certain conditions the only such ideals were hypersurface sections of Gorenstein ideals of codimension $3$ and deviation $2$.
Some of these assumptions were removed by Vasconcelos and Villarreal \cite{vv}. They showed that the result is true for ideals that are generically complete intersections. Still, the general case seemed elusive. Notably, Artinian examples are not covered by the Vasconcelos-Villarreal theorem.

Let $R$ be a local ring and let $I$ be a Gorenstein ideal of codimension $4$. Let 
$$\FF_\bullet :0\rightarrow R{\buildrel{d_4}\over\longrightarrow}F_3{\buildrel{d_3}\over\longrightarrow}F_2=F_2^*{\buildrel{d_3^*}\over\longrightarrow}F_3^*{\buildrel{d_4^*}\over\longrightarrow}R$$ be the minimal free resolution
 of $R/I$. 
 It was clear from the outset that the middle module $F_2$ has a nondegenerate symmetric bilinear form, thus the corresponding spinor group should play a role in the structure theory. 

In \cite{DGA-algebra}, \cite{KM1}, \cite{KM2}, Kustin and Miller proved several important structural results. In particular, they showed the existence of an associative, graded multiplicative multiplicative structure satisfying the Leibniz rule.
They also showed that there exists a basis of $F_2$ hyperbolic with respect to the quadratic form coming from this multiplication.

The next step was taken by Miles Reid \cite{reid}, who introduced spinor coordinates for the image of the map $d_3$ and conjectured that they generate the ideal $I$ up to radical. The role of these spinor coordinates was explained and extended in \cite{celikbas-laxmi-weyman} and in \cite{weyman-gorenstein} where this conjecture was proved.

In general, to classify ideals of a given codimension, one can consider an equivalence relation related to the concepts of specializations and deformations. A \it specialization \rm of an ideal $J$ is the image of $J$ in a quotient ring by a regular sequence $\underline \alpha$, which is still regular modulo $J$. If $I$ is a specialization of $J$, then $J$ is a \it deformation \rm of $I$. 
Two ideals that have a common deformation are in the same \it Herzog class \rm (see \cite{Herzog}, \cite{GNW3}). 

There was a lot of progress in recent years. A direct link was found in \cite{weyman-gorenstein} between the structure of Gorenstein ideals of codimenson $4$ with $n$ generators and a root system of type $E_n$. This suggests that the situation is very different for $n\le 8$ and for $n\ge 9$. One expects that for $n\le 8$ every Gorenstein ideal is in the linkage class of a complete intersection (\emph{licci}) and that they belong to only finitely many Herzog classes. For $n\ge 9$ one clearly has non-licci ideals, notably ``Tom and Jerry'' cases: the defining ideal of the embedding of ${\bf P}^1\times {\bf P}^1\times{\bf P}^1$ into ${\bf P}^7$ and the defining ideal of the embedding of ${\bf P}^2\times {\bf P}^2$ into ${\bf P}^8$. One also expects the existence of infinitely many Herzog classes of licci ideals for each $n \geq 9$.

The goal of the present paper is to give a gentle introduction to this circle of ideas. We concentrate mainly on the smallest case $n=6$. We explicitly construct in this case the so-called \emph{higher structure maps} introduced in \cite{weyman-gorenstein}. Such higher structure maps are already well-studied in the codimension three case \cite{W18}, \cite{Gue-Wey}, \cite{GNW1}, \cite{GNW3}. For a Gorenstein ideal of codimension four with $n \geq 6$ generators, these structure maps are related to the grading of the Kac-Moody Lie algebra $E_n$ induced by a particular simple root $\alpha_1$.

Studying the case $n=6$, we apply the theory of higher structure maps to prove the Vasconcelos-Villarreal theorem over complete regular local rings, without assuming the Gorenstein ideal is generically a complete intersection. In particular our result covers the Artinian case. The idea is to construct another family of higher structure maps related to a different grading of $E_6$, and then to mimic the method used by the second, third and fourth author for perfect ideals of codimension $3$ in \cite{Guerrieri-Ni-Weyman}.

The paper is organized as follows. In \S\ref{section:pre}, we recall basic material on Gorenstein ideals and their spinor structure, Lie algebras associated to the graph $E_n$, generic rings, and Schubert varieties. We then give a conceptual overview of the main argument using this language. Strictly speaking, a lot of this background is not mandatory for reading the later sections, but the computations therein would be much more opaque without this guiding motivation.

In \S\ref{section:hsm1} we introduce the higher structure maps related to the $\alpha_1$-grading. We give formulas for the first few graded components. For the case of $E_6$, we give a complete description of two critical representations. We also give the formulas for a split exact complex, with generic choice of defect variables. In \S\ref{section:hsm2} we present the higher structure maps related to the simple root $\alpha_2$. Again, we calculate them for a split exact complex and give formulas for few graded components.

In \S\ref{section:structure_theorems} we explain how to use both families of higher structure maps to prove the following structure theorem (see Theorem \ref{th:generic-gorenstein}) for Gorenstein ideals of codimension four with $6$ generators:
\begin{thm*}
Let $R$ be a complete regular local ring in which two is a unit or a graded polynomial ring over quadratically closed field of characteristic $\neq 2$. Let $I\subset R$ be a Gorenstein ideal of codimension four, minimally generated by six elements.

Then $I=(J,y)$, where $J$ is the ideal generated by submaximal pfaffians of some skew-symmetric $5\times 5$ matrix and $[y]$ is a regular element of $R/J$.    
\end{thm*}
\noindent Finally, in \S\ref{section:conjecture-end} we sketch similar results and conjectures for Gorenstein ideals of codimension four with $7$ and $8$ generators.

\section{Preliminaries}\label{section:pre}

In this section we give preliminaries on topics relevant to the rest of the paper. We start with the structure theory of Gorenstein ideals. Next we discuss Lie algebras and their representations. Then we briefly introduce a generic ring for the truncated resolutions of Gorenstein ideals of codimension four. Finally, we include some background on Schubert varieties in homogeneous spaces.

\subsection{Structure theory of Gorenstein ideals of codimension four}\label{section:pre1}

Let $R$ be a commutative Noetherian ring. We generally assume $R$ to be local or graded. We also assume $\frac{1}{2} \in R$. For a matrix $A$ with entries in $R$ we always denote by $I_d(A)\subset R$ the ideal generated by its $d\times d$ minors.

We will work with free resolutions of codimension four Gorenstein ideals of $ R$. They have the form
\begin{equation}\label{complexF}
\FF: 0 \longrightarrow R \buildrel{d_4}\over\longrightarrow F \buildrel{d_3}\over\longrightarrow  G \buildrel{d_2}\over\longrightarrow F^* \buildrel{d_1}\over\longrightarrow R.
\end{equation}
Here $F \cong R^{n}$ and $G \cong R^{2n-2}$, where $n$ is the minimal numbers of generators of the ideal. Sometimes we write $F^*=F_1$, $G=F_2$, $F=F_3$ and $F_0=F_4 \cong R$. We denote the basis of $F$ by $ \lbrace f_1, \ldots, f_n  \rbrace $. If $x \in F_i$ for $i \in \{0,\cdots,4\}$ we write $\deg(x) =i.$

A differential, graded commutative algebra structure exists on arbitrary minimal free resolutions (see \cite{Buchsbaum-Eisenbud_codim-3}), but it is not always associative. In our setting, the multiplication on $\mb{F}$ can be made strictly associative:

\begin{thm}[Theorem 4.3, \cite{DGA-algebra}]
Let $\mathbb{F}$ be the minimal free resolution of a Gorenstein ideal of codimension 4. 
Then there exists a differential, graded commutative, associative algebra structure on $\mathbb{F}$, i.e. a product map $\mathbb{F}\otimes\mathbb{F}\rightarrow\mathbb{F}$ such that
\begin{itemize}
	\item $F_iF_j\subset F_{i+j}$;
	\item $xy=(-1)^{\deg(x)\deg(y)}xy$;
	\item $d(xy)=d(x)y+(-1)^{\deg(x)}xd(y)$.
\end{itemize}
\end{thm}

We denote the graded pieces of this multiplicative structure by $m_{i,j}:F_i\otimes F_j\rightarrow F_{i+j}$. The multiplication $m_{2,2}:G\otimes G=F_2\otimes F_2\rightarrow F_4\simeq R$ gives a bilinear form on $G$ and so induces an isomorphism $G\simeq G^*$. Using this isomorphism, the resolution $\mathbb{F}$ can be put in a self-dual form:
$$
\mathbb{F}:0\rightarrow R\xrightarrow{d_1^*} F\xrightarrow{d_2^*} G^*\simeq G\xrightarrow{d_2} F^*\xrightarrow{d_1} R.
$$

In our context, the multiplication $m_{2,2}$ can be normalized.

\begin{thm}[Theorem 4.2, \cite{celikbas-laxmi-weyman}]\label{th:hyprbolic}
Assume that $R$ is a complete regular local ring in which 2 is invertible or a polynomial ring over an quadratically closed field $\mathbb{K}$ such that $\operatorname{char}\mathbb{K}\neq 2$. Then there exists a basis of the free module $G\simeq R^{2(n-1)}$ such that the multiplication $m_{2,2}:G\otimes G\rightarrow R$ is in hyperbolic form, i.e. its matrix is
$$
\begin{pmatrix}
	0&\operatorname{Id}_{n-1}\\
	\operatorname{Id}_{n-1}&0\\
\end{pmatrix}
$$
\end{thm}
\noindent We denote the hyperbolic basis of $G$ by $ \lbrace e_1, \ldots, e_{n-1}, \hat{e}_{1}, \ldots, \hat{e}_{n-1}  \rbrace$. This means that $e_i\cdot e_j=\hat{e}_{i}\cdot \hat{e}_{j}=0$ and $e_i\cdot\hat{e}_j=\delta_{i,j}$ for all $i,j=1,\cdots,n-1$.

Under the assumptions of Theorem \ref{th:hyprbolic} on the base ring $R$, there exists a \emph{spinor structure} on the resolution $\mathbb{F}$ (see \cite{celikbas-laxmi-weyman}). Let $\mathfrak{so}(G)\simeq\mathfrak{so}(2n-2)$ be the orthogonal Lie algebra and let $V$ be the fundamental representation of $\mathfrak{so}(2n-2)$ with the highest weight $\omega_{n-1}$. This is one of the half-spinor representations of the orthogonal Lie algebra. There exists a map $\widetilde{a_3}:R\rightarrow V\otimes R$ such that the following diagram commutes:
$$\begin{tikzcd}
R \arrow[rr,"S_2(\widetilde{a}_3)"]\arrow[rrd,"a_3"]&& S_2 V\otimes R\arrow[d,"\textbf{p}\otimes R"]\\
&& \bigwedge^{n-1}G
\end{tikzcd}$$
Here $a_3:R\rightarrow\bigwedge^{n-1}G$ is the structure map given by the First Structure Theorem of Buchsbaum and Eisenbud \cite{BE74}, while $\textbf{p}:S_2 V \rightarrow\bigwedge^{n-1} G$ is the unique $\mathfrak{so}(2n-2)$-equivariant map up to scale. The coordinates of the map $\widetilde{a}_3$ are called the \emph{spinor coordinates}. Let $[k]:=\{1,\cdots,k\}$, $K\subset [n-1] $ and let $J:=K\cup \left(2\cdot[n-1]\setminus2 K\right)$. It follows from the above diagram that the spinor coordinate ${\widetilde{a}_3}^*(J)$ is the square of the Buchsbaum-Eisenbud multiplier $a_3^*(K)$.
For concrete examples of computations of spinor coordinates see \cite{celikbas-laxmi-weyman} and Section \ref{section:hsm1} of this work.

\subsection{Lie algebras and representations}

For the remainder of \S\ref{section:pre}, we will work over $\mb{C}$. In particular, the ring $R$ is assumed to be a $\mb{C}$-algebra. We will drop this assumption in \S\ref{section:hsm1} onwards.

\subsubsection{Construction}\label{bg:lie-construction}
Fix integers $p,q,r \geq 1$, and let $T=T_{p,q,r}$ denote the graph
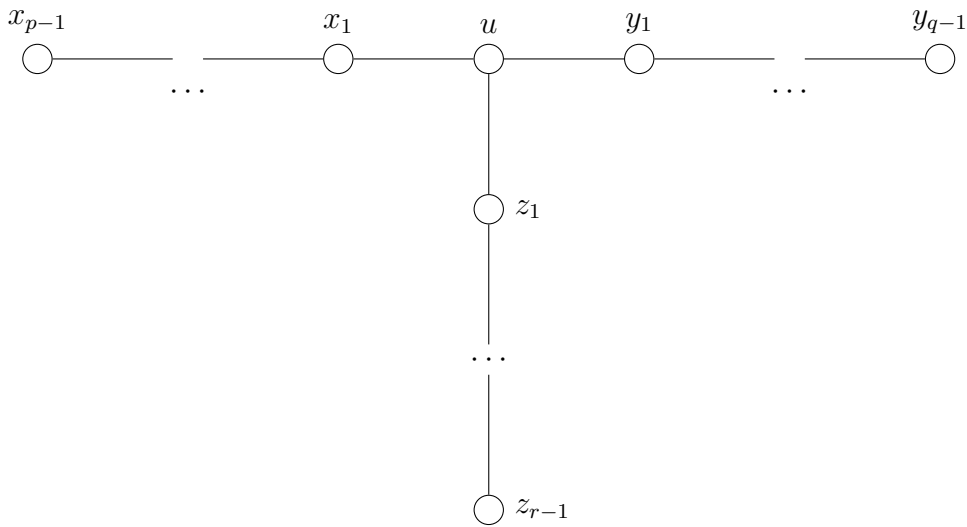
\begin{center}
		$\begin{tikzpicture}
			\node 
            [circle,fill=white,draw,label=above:$x_{p-1}$] (1) at (0,0) {};
			\node
		[circle,label=below:$\cdots$] (2) at (2,0) {};
			\node
            [circle,fill=white,draw,label=above:$x_1$] (3) at (4,0) {};
			\node
		[circle,fill=white,draw,label=above:$u$] (4) at (6,0) {};
			\node
		[circle,fill=white,draw,label=above:$y_1$] (5) at (8,0) {};
			\node
		[circle,label=below:$\cdots$] (6) at (10,0) {};
			\node
		[circle,fill=white,draw,label=above:$y_{q-1}$] (7) at (12,0) {};
                \node
            [circle,fill=white,draw,label=right:$z_{1}$] (8) at (6,-2) {};
                \node
		[circle,label=center:$\cdots$] (9) at (6,-4) {};
                \node
            [circle,fill=white,draw,label=right:$z_{r-1}$] (10) at (6,-6) {};

            \draw (1) to (2);
            \draw (2) to (3);
            \draw (3) to (4);
            \draw (4) to (5);
            \draw (5) to (6);
            \draw (6) to (7);
            \draw (4) to (8);
            \draw (8) to (9);
            \draw (9) to (10);

\end{tikzpicture}$
\end{center}

Let $n = p+q+r-2$ be the number of vertices. From the above graph, we construct an $n\times n$ matrix $A$, called the \emph{Cartan matrix}, whose rows and columns are indexed by the nodes of $T$:
\[
A = (a_{i,j})_{i,j \in T}, \quad a_{i,j} = \begin{cases}
	2 &\text{if $i = j$,}\\
	-1 &\text{if $i,j \in T$ are adjacent,}\\
	0 &\text{otherwise.}
\end{cases}
\]
$T$ is a Dynkin diagram if and only if $1/p + 1/q + 1/r > 1$; in this case we say it is of \emph{finite type}. We will primarily be interested in the diagrams $E_n \coloneqq T_{3,n-3,2}$, which are of finite type when $n \leq 8$. We next describe how to construct the associated Lie algebra $\mf{g}$.

Let $\mf{h} = \mb{C}^{2n - \rank A}$, and pick independent sets $\Pi = \{\alpha_i\}_{i\in T} \subset \mf{h}^*$ and $\Pi^\vee = \{\alpha_i^\vee\}_{i\in T} \subset \mf{h}$ satisfying the condition
\[
\langle \alpha_i^\vee,\alpha_j \rangle = a_{i,j}.
\]
The $\alpha_i$ are the \emph{simple roots} and the $\alpha_i^\vee$ are the \emph{simple coroots}. If $1/p + 1/q + 1/r = 1$, then $T = E_{n-1}^{(1)}$ is of \emph{affine type} and $\rank A = n-1$. Otherwise $\rank A = n$, and $\Pi, \Pi^\vee$ are bases of $\mf{h}^*,\mf{h}$ respectively.

The Lie algebra $\mf{g} := \mf{g}(T)$ is generated by $\mf{h}$ together with elements $e_i,f_i$ for $i\in T$, subject to the defining relations
\begin{gather*}
	[e_i,f_j] = \delta_{i,j} \alpha_i^\vee,\\
	[h,e_i] = \langle h, \alpha_i \rangle e_i, [h,f_i] = -\langle h,\alpha_i \rangle f_i  \text{ for } h \in \mf{h},\\
	[h,h'] = 0 \text{ for } h,h' \in \mf{h},\\
	\ad(e_i)^{1-a_{i,j}}(e_j) = \ad(f_i)^{1-a_{i,j}}(f_j) \text{ for } i \neq j.
\end{gather*}
Under the adjoint action of $\mf{h}$, the Lie algebra $\mf{g}$ decomposes into eigenspaces as $\mf{g} = \bigoplus \mf{g}_\alpha$, where
\[
\mf{g}_\alpha = \{x \in \mf{g} : [h,x] = \alpha(h)x \text{ for all } h \in \mf{h}\}.
\]
This is the \emph{root space decomposition} of $\mf{g}$.

Note that if $T$ is \emph{not} of affine type, then the Cartan matrix is invertible, and $\mf{g}$ is generated by $e_i$ and $f_i$ for $i \in T$ since $\alpha_i^\vee$ is a basis of $\mf{h}$. We will later only be interested in the case that $T$ is one of the finite type diagrams $E_n$ for $6 \leq n \leq 8$.

\subsubsection{Gradings on $\mf{g}$}\label{bg:lie-grading1}

Assume that $T$ is not of affine type. Let $Q \subset \mf{h}^*$ be the root lattice $\bigoplus_{i\in T} \mb{Z}\alpha_i$. If $\mf{g}_\alpha \neq 0$, then necessarily $\alpha \in Q$. If such an $\alpha$ is nonzero, we say it is a \emph{root}, and denote the set of all roots by $\Delta$. Hence the Lie algebra $\mf{g}$ is $Q$-graded. By singling out a vertex $t \in T$, this $Q$-grading can be coarsened to a $\mb{Z}$-grading by considering only the coefficient of $\alpha_t$. We refer to this as the $\alpha_t$-grading.

Write $h_i \in \mf{h}$ for the basis dual to the simple roots $\alpha_i \in \mf{h}^*$. The degree zero part of $\mf{g}$ in the $\alpha_t$-grading is
\[
\mf{g}^{(t)} \times \mb{C}h_t
\]
where $\mf{g}^{(t)}$ is the subalgebra generated by $\{e_i, f_i\}_{i \neq t}$ and $\mb{C}h_t$ is the one-dimensional abelian Lie algebra spanned by $h_t$. The decomposition of $\mf{g}$ into $\alpha_t$-graded components is just its decomposition into eigenspaces for the adjoint action of $h_t$:
\[
\mf{g} = \bigoplus_{j  \in \mb{Z}} \ker(\ad(h_t) - j).
\]

\begin{example}[$\alpha_2$-grading on $E_n$]\label{ex:2-grading}
	We label the simple roots of $E_n$ as follows:
	\begin{center}
		$
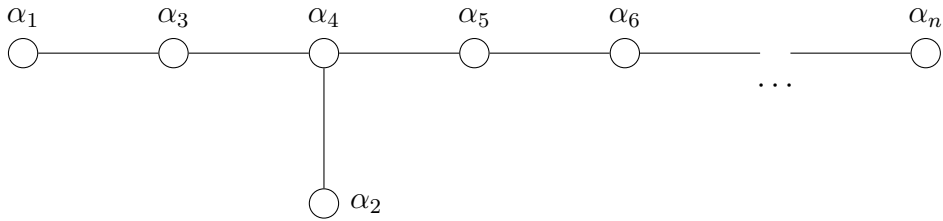
\begin{tikzpicture}
			\node
			[circle,fill=white,draw,label=above:$\alpha_1$] (1) at (0,0) {};
			\node
			[circle,fill=white,draw,label=above:$\alpha_3$] (2) at (2,0) {};
			\node
			[circle,fill=white,draw,label=above:$\alpha_4$] (3) at (4,0) {};
			\node
			[circle,fill=white,draw,label=right:$\alpha_2$] (4) at (4,-2) {};
			\node
			[circle,fill=white,draw,label=above:$\alpha_5$] (5) at (6,0) {};
			\node
			[circle,fill=white,draw,label=above:$\alpha_6$] (6) at (8,0) {};
			\node
			[circle,label=below:$\cdots$] (7) at (10,0) {};
			\node
			[circle,fill=white,draw,label=above:$\alpha_n$] (8) at (12,0) {};
			\draw (1) to (2);
			\draw (2) to (3);
			\draw (3) to (4);
			\draw (3) to (5);
			\draw (5) to (6);
			\draw (6) to (7);
			\draw (7) to (8);
		\end{tikzpicture}$
	\end{center}
	Let $F_1 = \mb{C}^{n}$ and $F_4 = \mb{C}^1$. (This notation will be clarified in \S\ref{section:hsm2}.) Then $\mf{g}^{(2)} = \mf{sl}(F_1) = \mf{sl}(F_1) \times \mf{sl}(F_4)$.
	
	For $n \leq 8$, the $\alpha_2$-graded decomposition of $E_n$ is
	\begin{align*}
		3 &\hspace{1in} S_{2,1^7} F_1 \otimes S_3 F_4^*\\
		2 &\hspace{1in} \bigwedge^6 F_1 \otimes S_2 F_4^*\\
		1 &\hspace{1in} \bigwedge^3 F_1 \otimes F_4^*\\
		0 &\hspace{1in} \mf{sl}(F_1) \times \mf{sl}(F_4) \times \mb{C}h_2\\
		-1 &\hspace{1in} \bigwedge^3 F_1^* \otimes F_4\\
		-2 &\hspace{1in} \bigwedge^6 F_1^* \otimes S_2 F_4\\
		-3 &\hspace{1in} S_{2,1^7} F_1^* \otimes S_3 F_4
	\end{align*}
	Note that $S_{2,1^7} F_1 = 0$ for $n \leq 7$, so there are only components in degrees $\pm 3$ for $n=8$. We have displayed each graded component as a representation of $\mf{gl}(F_1) \times \mf{gl}(F_4)$ using the map
	\[
		\mf{gl}(F_1) \times \mf{gl}(F_4) \to \mf{sl}(F_1) \times \mf{sl}(F_4) \times \mb{C}h_2
	\]
	which sends traceless matrices to themselves, the identity matrix in $\mf{gl}(F_1)$ to $3h_2$, and the identity matrix in $\mf{gl}(F_4)$ to $-h_2$.
\end{example}
\begin{example}[$\alpha_1$-grading on $E_n$]\label{ex:1-grading}
	Using the same labeling as in Example~\ref{ex:2-grading}, we have that $\mf{g}^{(1)} = \mf{so}(G)$ where $G = \mb{C}^{2n-2}$. Let $V$ denote the half-spinor representation of $\mf{so}(G)$ associated to the vertex $2 \in T$. If $n$ is odd, then $V$ is self-dual; otherwise it is dual to the other half-spinor representation.
	
	The $\alpha_1$-graded decompositions of $E_n$ for $6 \leq n \leq 8$ are
	\begin{align*}
		E_6 &= 0 \oplus V \oplus (\mf{so}(G) \times \mb{C}h_1) \oplus V^* \oplus 0\\
		E_7 &= \mb{C} \oplus V \oplus (\mf{so}(G) \times \mb{C}h_1) \oplus V^* \oplus \mb{C}\\
		E_8 &= G \oplus V \oplus (\mf{so}(G) \times \mb{C}h_1) \oplus V^* \oplus G
	\end{align*}
	where the graded components are displayed from degree $-2$ to $2$.
\end{example}

\subsubsection{Representations}\label{bg:reps}
Let $V$ be a representation of $\mf{g}$. For $\lambda \in \mf{h}^*$, define the \emph{$\lambda$-weight space of $V$} to be
\[
V_\lambda = \{v \in V : h v = \lambda(h)v \text{ for all }h \in \mf{h}\}.
\]
If $V_\lambda \neq 0$, then we say $\lambda$ is a \emph{weight} of $V$. A nonzero vector $v \in V_\lambda$ is a \emph{highest weight vector} if $e_i v = 0$ for all $i$. If such a $v$ generates $V$ as a $\mf{g}$-module, then we say $V$ is a \emph{highest weight module} with highest weight $\lambda$.

For each $\lambda$, there is a unique irreducible highest weight module with highest weight $\lambda$, which we denote by $V(\lambda)$ or $V(\lambda,\mf{g})$. If $\mf{g}$ is of finite type, then $V(\lambda)$ is finite-dimensional exactly when $\lambda$ is a dominant integral weight, i.e. a nonnegative integral combination of the fundamental weights $\omega_i \in \mf{h}^*$ dual to the simple coroots. Given a subalgebra $\mf{g}^{(t)}$ of $\mf{g}$, we write $V(\lambda,\mf{g}^{(t)})$ for the $\mf{g}^{(t)}$-representation generated by the highest weight space in $V(\lambda)$. Explicitly, if $\lambda = \sum c_i \omega_i$, the representation $V(\lambda,\mf{g}^{(t)})$ is the irreducible highest weight module of $\mf{g}^{(t)}$ with highest weight $\sum_{i \neq t} c_i \omega_i$.

We have already discussed how a vertex $t \in T$ induces a $\mb{Z}$-grading on $\mf{g}$, decomposing it into eigenspaces for the action of $h_t$. The representations $V(\lambda)$ also decompose into eigenspaces for the action of $h_t$, with eigenvalues $\langle h_t, \lambda \rangle, \langle h_t, \lambda \rangle-1,\ldots$, terminating iff $V(\lambda)$ is finite-dimensional. These eigenvalues may not be integers; to get a $\mb{Z}$-grading we often translate all degrees by $-\langle h_t, \lambda \rangle$ so that the top graded component is in degree 0.

\subsubsection{Connection to generic free resolutions}\label{subsec:motivation-1}
The Lie algebras defined in \S\ref{bg:lie-construction} appear in the study of generic free resolutions, which we now briefly survey. The classical Hilbert-Burch theorem gives a structure theorem for free resolutions of the form
\[
	0 \to R^{n-1} \to R^n \to R.
\]
In fact, it gives the universal such resolution, which specializes uniquely to any other free resolution with the same ranks of free modules. We refer to the sequence of ranks as the \emph{format} of the complex, thus Hilbert-Burch describes the universal free resolution of format $(1,n,n-1)$.

The situation for other length 2 formats $(f_0,f_1,f_2)$ with $f_0 > 1$ was studied by Hochster \cite{H75}, Huneke \cite{huneke}, and many others. In short, a universal free resolution also exists in these cases, and it essentially encodes the first structure theorem of Buchsbaum and Eisenbud, of which Hilbert-Burch is a special case.

However, for formats of length 3 and beyond, Bruns showed that universal free resolutions cannot exist \cite{bruns}. It is necessary to drop the requirement that the specialization is unique. We say that a particular free resolution $\mb{F}^{\mathrm{gen}}$ is \emph{generic} for its format if it specializes to any other resolution of the same format, and we refer to its base ring $R_{\mathrm{gen}}$ as a \emph{generic ring}.

If we have a homomorphism $w\colon R_{\mathrm{gen}}\to R$ specializing $\mb{F}^{\mathrm{gen}}$ to some resolution $\mb{F}$ over $R$, we would like to interpret $w$ as encoding ``higher structure maps'' for $\mb{F}$, in the same spirit as Hilbert-Burch or Buchsbaum-Eisenbud for the length 2 case. To do this, an explicit description of the generic ring is necessary. Note that unlike universal free resolutions, generic ones are not unique for a given format. For resolutions of length 3, a particular generic pair $(\Rgen,\Fgen)$ was constructed in \cite{Weyman89} and \cite{W18}. One of the main results of this study is that the Lie algebra $\mf{g}(T_{p,q,r})$ acts on the generic ring $\Rgen$ for the format $(p-1, p+q,q+r,r-1)$. Accordingly, the explicit higher structure maps extracted from this machinery are closely related to the representation theory of $\mf{g}$.

We have evidence to suggest that a similar situation holds for resolutions of codimension 4 Gorenstein ideals, i.e. self-dual resolutions of format $(1,n,2n-2,n,1)$. More precisely, we introduce two families of higher structure maps for such resolutions in \S\ref{section:hsm1} and \S\ref{section:hsm2}. While we give explicit ad-hoc definitions for both, the former family of higher structure maps actually has a general systematic construction (akin to that for $\Rgen$ in the length 3 case), but we lack such a construction for the latter family. These higher structure maps are related to the representation theory of $E_n$.

\subsubsection{Generic ring}\label{section:pre3}

Now we describe the construction of a generic ring  for the complex $\mathbb{F}'$, which is the resolution of a Gorenstein ideal of codimension 4 truncated at $F_1$. This is related to higher structure maps in the $\alpha_1$-grading (see \S\ref{section:hsm1}). To keep our notation consistent with the notes \cite{weyman-gorenstein}, we use the convention:
\begin{equation}\label{truncated}
\mathbb{F}':0\rightarrow R\xrightarrow{d_1^*} F\xrightarrow{d_2^*} G\xrightarrow{Qd_2} F^*
\end{equation}
We have a fixed identification $G\simeq G^*$ using a bilinear symmetric form $Q=(\cdot,\cdot)$. Let $\operatorname{Grass}(1,F)$ be the Grassmanian of lines in $F$, i.e. the projective space $\mathbb{P}(F)$, and let $\operatorname{IGrass}(n-1,G)$ be the Grassmanian of maximal isotropic subspaces of $G$. We have the corresponding tautological sequences
$$
0\rightarrow\mathcal{R}\rightarrow F\times\operatorname{Grass}(1,F)\rightarrow\mathcal{Q}\rightarrow 0,
$$
$$
0\rightarrow\mathcal{S}\rightarrow G\times\operatorname{IGrass}(n-1,G)\rightarrow\mathcal{S}^*\rightarrow 0.
$$
We define the first approximation of the generic ring to be $$A_1(n):=H^0\left(\operatorname{Grass}(1,F)\times\operatorname{IGrass}(n-1,G),\mathcal{O}_Z\right)$$ Here $\mathcal{O}_Z$ is the sheaf defined as
$$\mathcal{O}_Z:=\bigoplus_{a,b\geq 0, \lambda\vdash (n-1)}S_a\mathcal{R}^*\otimes S_\lambda\mathcal{Q}\otimes S_\lambda\mathcal{S}^*\otimes\mathcal{O}(b)/\mathcal{I},$$
where the ideal sheaf $\mathcal{I}$ is generated by $\bigwedge^{n-1}\mathcal{Q}\otimes\bigwedge^{n-1}\mathcal{S}^*$ and $\mathcal{R}^*\otimes\mathcal{O}(2)$. By Bott's theorem, higher cohomology groups of $\mathcal{O}_Z$ vanish, while $A_1(n)$ decomposes into irreducibles as
$$
\bigoplus_{a,b,\lambda}S_{(\lambda_1,\cdots,\lambda_{n-2},0,-a)}F \otimes V\Big((\lambda_1-\lambda_2)\omega_1+\cdots+(\lambda_{n-3}-\lambda_{n-2})\omega_{n-3}+b\omega_{n-2}+\lambda_{n-2}(\omega_{n-2}+\omega_{n-1}),D_{n-1}\Big),
$$
where $(\lambda_1,\cdots,\lambda_{n-1})=\lambda\vdash(n-1)$.

By construction, over the ring $A_1(n)$ there is a natural self-dual complex
$$
\mathbb{F}^1:0\rightarrow A_1(n)\rightarrow F\otimes A_1(n)\rightarrow G\otimes A_1(n)\simeq G^*\otimes A_1(n)\rightarrow F^*\otimes A_1(n)\rightarrow A_1(n)
$$ 
The homology groups $H_3(\mathbb{F}^1)$ and $H_4(\mathbb{F}^1)$ are zero. Similar to codimension three case, one performs a procedure of killing cycles in certain complexes constructed from $\mathbb{F}^k:=\mathbb{F}^{k-1}\otimes_{A(n)_{k-1}}A(n)_{k}$, $k\geq 1$, obtaining a sequence of rings $A(n)_k\hookrightarrow A(n)_{k+1}$. The union $A(n)_\infty:=\bigcup A(n)_k$ is a generic ring for the complex $\mathbb{F}'\otimes A(n)_\infty$ of length three. Thus the only non-zero homology groups of $\mathbb{F}^\infty:=\mathbb{F}^1\otimes_{A(n)_1} A(n)_\infty$ are $H_0(\mathbb{F}^\infty)$ and $H_1(\mathbb{F}^\infty)$.

Consider defects of the lifts required for the construction of each $A(n)_k$, i.e. modules $\mathbb{L}_k$ parametrizing different choices in the procedure of killing cycles. The ring $A(n)_\infty$ has a natural structure of a multiplicity-free representation of the defect Lie algebra $\mathbb{L}=\bigoplus_{k\geq 1}\mathbb{L}_k$. The algebra $\mathbb{L}$ is the positive part of the Lie algebra $\mathfrak{g}(A_{n-1})\oplus\mathfrak{g}(T_{n-2,2,3})=\mathfrak{sl}(n)\oplus\mathfrak{g}(E_n)$ in the grading induced by the root corresponding to the vertex $\alpha_1$ of $E_n$. The Lie algebra $\mathfrak{sl}(n)$ is identified with $\mathfrak{sl}(F)$. The decomposition of the generic ring $A(n)_\infty$ into irreducible representations is
$$
\bigoplus_{a,b,\lambda}S_{(\lambda_1,\cdots,\lambda_{n-2},0,-a)}F \otimes V\Big(a\omega_1+(b+\lambda_{n-2})\omega_2+\lambda_{n-2}\omega_3+(\lambda_{n-3}-\lambda_{n-2})\omega_{4}+\cdots+(\lambda_1-\lambda_2)\omega_n,E_n\Big)
$$
As always, we follow the Bourbaki convention for labeling the nodes. 

The ring $A(n)_\infty$ is generated by three \emph{critical representations}: $W_1:=F\otimes V(\omega_1,E_n)$, $W_0:=\mathbb{C}\otimes V(\omega_2,E_n)$ and $W_n:=F^*\otimes V(\omega_n,E_n)$. By construction, for every resolution of the form (\ref{truncated}), defined over a commutative ring $R$, there is a map $A(n)_\infty\rightarrow R$. Using this map, all the graded components of the critical representations are mapped to $R$ and this induces sequences of $R$-linear map called \emph{higher structure maps}. This is analogous to the situation in codimension three (see \cite{W18,Gue-Wey,GNW1}). We give explicit definitions of higher structure maps in Section \ref{section:hsm1}.

For example, the lowest component of $V(\omega_1,E_n)$ is the trivial representation of $\mathfrak{so}(G)$ and the map living in $F\otimes \mathbb{C}\simeq\operatorname{Hom(F,\mathbb{C})}$ is just the first differential. The lowest component of $V(\omega_n,E_n)$ is the standard representation $V(\omega_1,D_{n-1})=G$ and the map living in $F\otimes G\simeq F\otimes G^*\simeq \operatorname{Hom}(G,F)$ is the second differential. Finally, the lowest graded piece of $V(\omega_2,E_n)$ is the spinor representation $V(\omega_{n-1},D_{n-1})$ and the appropriate structure map is the spinor coordinate, living inside $\mathbb{C}\otimes V(\omega_{n-1},D_{n-1})\simeq\operatorname{Hom}\left(\mathbb{C},V(\omega_{n-1},D_{n-1})\right)$.

The proper \emph{higher} structure maps appear in the subsequent graded pieces of the critical representations. For three Dynkin cases of $n=6,7,8$ generators, there are finitely many graded components, since the corresponding graphs are of finite type. In the following section we describe all structure maps for $n=6$, i.e. for Gorenstein ideals with six generators. In general, we see infinitely many graded components and hence an infinite sequence of higher structure maps for ideals with $n\geq 9$ generators.

\subsubsection{Connection between gradings and higher structure maps}\label{subsec:motivation-2}

The gradings defined in \S\ref{bg:lie-grading1} play an important role in interpreting higher structure maps. We return to the length 3 setting to demonstrate. Fix a format $(1,2+q,q+r,r-1)$ and let $(\Rgen,\Fgen)$ be the associated generic pair as in \S\ref{subsec:motivation-1}. The Lie algebra $\mf{g}$ associated to $T = T_{2,q,r}$ acts on $\Rgen$. For convenience of exposition, let us assume that this algebra is of finite type. The dual of the representation $V(\omega_{x_1})$ appears inside $\Rgen$, so if we have a map $w\colon \Rgen \to R$ specializing $\Fgen$ to some resolutions $\mb{F}$, we may restrict to get a map
\[
	w^{(1)}\colon V(\omega_{x_1})^\vee \otimes R \to R.
\]
Let $F_1 = \mb{C}^{2+q}$ and $F_3 = \mb{C}^{r-1}$, so that $\mf{g}^{(z_1)} = \mf{sl}(F_1)\times \mf{sl}(F_3)$. If one decomposes $V(\omega_{x_1})^\vee$ with respect to the $\alpha_{z_1}$-grading, we can rewrite the above as
\[
	w^{(1)}\colon [F_1 \oplus (\bigwedge^3 F_1 \otimes F_3^*) \oplus \cdots] \otimes R \to R
\]
where $F_1$ is viewed as residing in degree 0. The restriction $w^{(1)}_0$ of this map to the bottom graded component recovers the first differential of $\mb{F}$, and by analyzing relations in the ring $\Rgen$, one can show that the restriction $w^{(1)}_1$ to the next graded component yields (part of) a multiplicative structure on $\mb{F}$. There are other structure maps $w^{(i)}$ that are defined by restricting $w$ to other representations inside of $\Rgen$, but we will omit discussing them for the moment.

Since we do not have a generic ring for length 4 resolutions, we will instead take an approach ``inverse'' to the above: by examining the decomposition of appropriate representations $V_i$, we attempt to define maps $w^{(i)}\colon V_i \otimes R \to R$, one graded component $w^{(i)}_j$ at a time, so that the whole map exhibits similar properties as the higher structure maps do in the length 3 case. This is admittedly ad-hoc, but for $E_6$ it is at least fairly tractable.

There are two main desirable properties of the higher structure maps in the length 3 case that we would like to emulate. The precise statements are somewhat technical, but for the map $w^{(1)}$, they can be summarized as follows:
\begin{enumerate}
	\item The entries of the map $w^{(1)}$ satisfy the Pl\"ucker relations, i.e. its symmetric square vanishes on all subrepresentations of $S_2 V(\omega_{x_1})^\vee$ except for the irreducible subrepresentation $V(2\omega_{x_1})^\vee$.
	\item If $T$ is of finite type and $\mb{F}^*$ is acyclic then the map $w^{(1)}$ is surjective.
	\item The bottom graded component of $w^{(1)}$ recovers the first differential of $\mb{F}$.
\end{enumerate}
We will use points (3) and (1) as guides for defining higher structure maps in the codimension 4 Gorenstein setting---loosely speaking, the bottom graded components should come directly from the resolution, and higher components should satisfy some relations that enable their computation in terms of lower ones already defined.

Points (1) and (2) together imply that $w^{(1)}$ defines a map from $\operatorname{Spec}R$ to a certain homogeneous space $G/P$ in its Pl\"ucker embedding. To explain this, we will now give some background on the relevant objects.

\subsection{$G/P$ and Schubert varieties}\label{sec:bg-schubert}
We assume throughout that $T$ is a Dynkin diagram and thus $\mf{g}$ is of finite type. There is a unique simply connected Lie group $G$ associated to the Lie algebra $\mf{g}$, and the representations of $G$ correspond to those of $\mf{g}$. For a fundamental weight $\omega_t$, the action of $G$ on the highest weight line in $\mb{P}(L(\omega_t))$ has stabilizer $P_t^+$, the subgroup of $G$ corresponding to a maximal parabolic subalgebra $\mf{p}_t^+$. Hence the orbit of this highest weight line can be identified with the homogeneous space $G/P_t^+$. For Dynkin type $A_n$ with the standard labeling of vertices, this construction produces the Grassmannian $\Gr(t,n+1)$. Accordingly, the reader may think of $G/P_t^+$ as a ``generalized Grassmannian.''

\subsubsection{Algebraic definition of the homogeneous space $G/P$}
Pick a vertex $t \in T$, let $\omega_t$ be the corresponding fundamental weight, and $V(\omega_t)$ the irreducible representation with highest weight $\omega_t$. Let $\mf{A} = \bigoplus_{n\geq 0} V(n\omega_t)^\vee$. This is a graded $\mb{C}$-algebra generated by $V(\omega_t)^\vee$ in degree 1:
\[
\bigoplus_{n\geq 0} V(n\omega_t)^\vee = (\Sym V(\omega_t)^\vee) / I_\Plucker.
\]
The ideal $I_\Plucker$ is comprised of subrepresentations vanishing on a highest weight vector $v \in V(\omega_t)$. We define $G/P_{t}^+ = \operatorname{Proj} \mf{A}$ to be the corresponding projective variety in $\mb{P}(V(\omega_t))$; this is the Pl\"ucker embedding of $G/P_{t}^+$.

\subsubsection{Weyl group and subgroups}
Let $W$ denote the Weyl group associated to $T$. It is generated by the  \emph{simple reflections} $\{s_i\}_{t\in T}$. Explicitly,
\[
W = \langle \{s_i\}_{i\in T} \mid (s_i s_j)^{m_{ij}} = 1\rangle
\]
where $m_{ij} = 1$ if $i=j$, $m_{ij} = 2$ if $i,j \in T$ are not adjacent, and $m_{ij} = 3$ if $i,j \in T$ are adjacent. The group $W$ is finite if and only if $T$ is a Dynkin diagram.

The Weyl group acts on $\mf{h}^*$: the simple reflections act via
\[
s_i(\lambda) = \lambda - \langle \alpha_i^\vee, \lambda\rangle \alpha_i.
\]
For any $t \in T$, the expression
\[
\exp(f_t)\exp(-e_t)\exp(f_t)
\]
defines an automorphism of any representation on which the actions of $e_t$ and $f_t$ are locally nilpotent. This includes the adjoint representation of $\mf{g}$ and $V(\lambda)$ for any dominant integral $\lambda$. By abuse of notation, we will also denote this automorphism by $s_t$, although it is really a (non-unique) lift thereof. For any element of the Weyl group, we may define an analogous automorphism by expressing $\sigma$ as a product of simple reflections. The resulting automorphism is not unique, but this is not important for our purposes.

A \emph{word} for $\sigma$ is a sequence of simple reflections whose product is $\sigma$. It is \emph{reduced} if there are no shorter words for $\sigma$. The \emph{length} $\ell(\sigma)$ is the length of a reduced word for $\sigma$. There is a partial order, called the \emph{(strong) Bruhat order} on $W$, defined so that $\sigma \geq \sigma'$ if a reduced word for $\sigma$ contains a reduced word for $\sigma'$ as a (not necessarily consecutive) substring. If this holds for some reduced word for $\sigma$, it in fact holds for all reduced words.

If $t \in T$, we let $W_t \subset W$ denote the subgroup generated by all simple reflections other than $s_t$. This is the stabilizer of the fundamental weight $\omega_t$ under the action of $W$. We write $W^t$ for the set of minimal length representatives of $W/W_t$.

\subsubsection{Pl\"ucker coordinates and Schubert cells}
Elements of $V(\omega_t)^\vee$ are called \emph{Pl\"ucker coordinates}. Let $p_e$ denote a lowest weight vector of $V(\omega_t)^\vee$ and let $p_\sigma = \sigma p_e$. The set $\{p_\sigma : \sigma\in W^{P_t}\}$ is the set of \emph{extremal Pl\"ucker coordinates}; they are defined up to scale. The representation $V(\omega_t)^\vee$ may have weights other than those in the $W$-orbit of $-\omega_t$, and thus there may be non-extremal Pl\"ucker coordinates. (The representation is \emph{miniscule} if all weights belong to the same $W$-orbit, in which case all Pl\"ucker coordinates are extremal.)

Let $v \in \mb{P}(V(\omega_t))$ denote the highest weight point, which we also think of as the point $P_t^+ / P_t^+ \in G/P_t^+$. By abuse of notation we will sometimes use $v$ to refer to a highest weight vector in $V(\omega_t)$ instead. Let $B^+$ denote the (positive) Borel subgroup of $G$. The point $v$ is the \emph{Borel-fixed point} of $G/P_t^+$. For $\sigma \in W^t$, the point $\sigma v$ is a \emph{torus-fixed point} of $G/P_t^+$.

Let $B^-$ denote the opposite (negative) Borel, and let $w \in W^t$. We define:
\begin{itemize}
	\item the Schubert cell $C_w = B^+ w v = B^+ w P_t^+/P_t^+$ and the Schubert variety $X_w = \overline{C_w}$,
	\item the opposite Schubert cell $C^w = B^- w v = B^- w P_t^+/P_t^+$ and the opposite Schubert variety $X^w = \overline{C^w}$.
\end{itemize}

The extremal Pl\"ucker coordinate $p_\sigma$ vanishes on $C_w$ (equivalently $X_w$) if and only if $\sigma \not\leq w$ in the Bruhat order. The coordinate $p_\sigma$ vanishes on $C^w$ (equivalently $X^w$) if and only if $\sigma \not\geq w$ in the Bruhat order. The dimension of $X_w$ is the dimension of $C_w$, which is $\ell(w)$.

All Schubert varieties are linearly defined, meaning that their homogeneous ideals in $\Sym V(\omega_t)^\vee$ are generated by the Pl\"ucker relations and some Pl\"ucker coordinates. The ideal may contain non-extremal Pl\"ucker coordinates, however.

\begin{example}\label{ex:Schubert-X^w}
	Recall that we are assuming $T = T_{p,q,r}$ to be a Dynkin diagram. Let $t = x_{p-1}$ and let $w = s_{z_1} s_u s_{x_1} \cdots s_{x_{p-1}} \in W^{x_{p-1}}$. The opposite Schubert variety $X^w$ inside of $G/P_{x_{p-1}}^+$ has codimension $\ell(w) = p+1$.
	
Modulo the Pl\"ucker relations, the ideal of $X^w$ is generated by the extremal Pl\"ucker coordinates
	\[
		p_e, p_{s_{x_{p-1}}}, p_{s_{x_{p-2}}s_{x_{p-1}}}, \ldots, p_{s_{y_{q-1}} \cdots s_{y_1} s_u s_{x_1} \cdots s_{x_{p-1}}}
	\]
	where the subscripts are all suffixes of the word $s_{y_{q-1}} \cdots s_{y_1} s_u s_{x_1} \cdots s_{x_{p-1}}$. These Pl\"ucker coordinates span the bottom $\alpha_{z_1}$-graded component of the representation $V(\omega_{x_{p-1}})^\vee$. There are no non-extremal Pl\"ucker coordinates in the ideal of $X^w$ in this case.
	
The opposite Schubert variety $X^w$ is the union of all opposite Schubert cells $C^\sigma$ with $\sigma \geq w$. Every opposite Schubert variety has an action of $B^-$ by construction, but particular ones may have more symmetry. In this case, there is an action of the negative maximal parabolic $P_{z_1}^-$ on $X^w$. Since this parabolic contains $B^-$, the orbits under this action are unions of opposite Schubert cells. Specifically, $C^\sigma$ and $C^{\sigma'}$ are contained in the same $P_{z_1}^-$-orbit exactly when $\tau \sigma = \sigma'$ for some $\tau \in W_{z_1}$. Thus the $P_{z_1}^-$-orbits in $G/P_{x_{p-1}}^+$ are indexed by the double cosets $W_{z_1} \backslash W / W_{x_{p-1}}$, and $X^w$ is the union of all orbits except for the one corresponding to the identity $e \in W$, which is the open orbit.
	
In particular, for any $q \in X^w$, the local ring of $X^w$ at $q$ is isomorphic to the local ring of $X^w$ at a torus-fixed point $\sigma v$ where $\sigma \in W$ is a minimal length representative of its double coset in $W_{z_1} \backslash W / W_{x_{p-1}}$. Furthermore, this local ring is a cone over the local ring of the Kazhdan-Luztig variety $\mathcal{N}_\sigma^w \coloneqq X^w \cap C_\sigma$ at $\sigma v$.
\end{example}

\subsubsection{Geometric interpretation of higher structure maps}\label{subsec:motivation-3}

Let $I \subset R$ be a codimension 4 Gorenstein ideal, and let $\mb{F}$ be a resolution of $R/I$, with format $(1,n,2n-2,n,1)$. Let $T = E_n = T_{3,n-3,2}$ and let $\mf{g}$ be the associated Lie algebra. We use the numbering of vertices as in Example~\ref{ex:2-grading}. Assume that $T$ is a Dynkin diagram, i.e. $n \leq 8$. Our main goal is to define higher structure maps for $\mb{F}$ with properties emulating those given at the end of \S\ref{subsec:motivation-2}. In geometric language, these will be as follows:

\begin{enumerate}
	\item The entries of $w^{(1)}$ satisfy the Pl\"ucker relations, so $w^{(1)}$ defines a map from $\operatorname{Spec} R$ to the affine cone over $G/P_{1}^+$ in its Pl\"ucker embedding.
	\item $w^{(1)}$ is surjective, so the above also gives a well-defined map $f\colon \operatorname{Spec} R \to G/P_{1}^+$.
	\item The bottom $\alpha_2$-graded component of $w^{(1)}$ recovers $I$, so if we let $w = s_2 s_4 s_3 s_1$ be as in Example~\ref{ex:Schubert-X^w}, then $f^{-1} X^w = \operatorname{Spec} R/I$ as a closed subscheme of $\operatorname{Spec} R$.
\end{enumerate}

To be more precise, we will define \emph{two} families of higher structure maps in \S\ref{section:hsm1} and \S\ref{section:hsm2}, corresponding to the vertices $1$ and $2$ on the diagram $E_n$ respectively. The aforementioned $w^{(1)}$ is one of the maps we define in \S\ref{section:hsm2}. Condition (3) will be automatic from our construction. To inductively define higher components of $w^{(1)}$, it will be necessary to define some other structure maps in that section as well. Since all resolutions are split exact on a dense open set, it suffices to confirm (1) for higher structure maps on a split exact complex, which is why most of the calculations in the subsequent sections will be done in this setting.

For (2), we will show that $w^{(1)}$ gives a well defined map to $G/P_{x_1}^+$ by showing that it descends from a map $\operatorname{Spec} R \to G$. The key point is that $G$ is a linear algebraic group, in particular affine, so by algebraic Hartogs any map $\operatorname{Spec} R - V(I) \to G$ must extend to the entirety of $\operatorname{Spec} R$ since $I$ has codimension $4 \geq 2$. As such, it suffices to construct the desired map locally on $\operatorname{Spec} R - V(I)$, on which $\mb{F}$ is split exact. The only concern is that, working locally on this set, the pieces might not glue together. It turns out that the gluing works if we fix choices of \emph{both} families of higher structure maps for $\mb{F}$, so this is where the maps from \S\ref{section:hsm1} are needed. (Our actual argument is more algebraic: we exhibit $w^{(1)}$ as a row of an invertible square matrix, from which it follows that $w^{(1)}$ must be surjective. We construct this square matrix on two different localizations of $R$ and then argue that they must agree. But this square matrix is really the action of some $R$-point of $G$ on a representation, so this is effectively the same argument in different language.)

If $R$ is local, then by combining (2) and (3) with Example~\ref{ex:Schubert-X^w}, we conclude that $R/I$ is a specialization of one of the local rings $\mathcal{O}_{\mathcal{N}_\sigma^w, \sigma v}$ where $\sigma \in W$ is a minimal length representative of its double coset in $W_{2} \backslash W / W_{1}$.

\section{Higher structure maps from the grading induced by $\alpha_1$}\label{section:hsm1}

In this and in the next section, we consider higher structure maps with respect to the grading induced by the nodes of the following diagram $E_n$ (see Example \ref{ex:2-grading}).

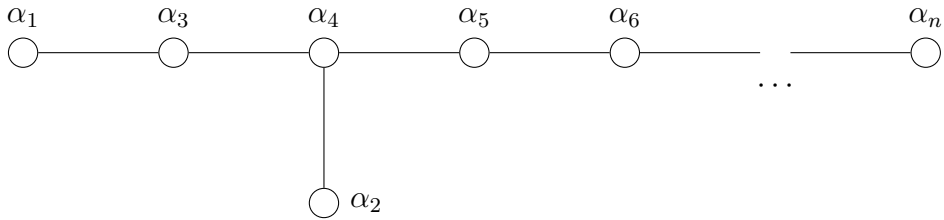
\begin{center}
		$\begin{tikzpicture}
			\node
			[circle,fill=white,draw,label=above:$\alpha_1$] (1) at (0,0) {};
			\node
			[circle,fill=white,draw,label=above:$\alpha_3$] (2) at (2,0) {};
			\node
			[circle,fill=white,draw,label=above:$\alpha_4$] (3) at (4,0) {};
			\node
			[circle,fill=white,draw,label=right:$\alpha_2$] (4) at (4,-2) {};
			\node
			[circle,fill=white,draw,label=above:$\alpha_5$] (5) at (6,0) {};
			\node
			[circle,fill=white,draw,label=above:$\alpha_6$] (6) at (8,0) {};
			\node
			[circle,label=below:$\cdots$] (7) at (10,0) {};
			\node
			[circle,fill=white,draw,label=above:$\alpha_n$] (8) at (12,0) {};

                \draw (1) to (2);
			\draw (2) to (3);
			\draw (3) to (4);
			\draw (3) to (5);
			\draw (5) to (6);
			\draw (6) to (7);
			\draw (7) to (8);
\end{tikzpicture}$
\end{center}
In this section, we look specifically at the grading induced by the node $\alpha_1$.

\subsection{First graded components}\label{sec:firstcomponents,alpha1}

The higher structure maps in the first graded components of the critical representations $W_1$ and $W_{n}$ can be computed by lifting cycles in some appropriate exact complexes. 
We give an interpretation of these cycles in terms of the spinor coordinates, valid for all $n$. The computations that we will perform in this section are similar to those obtained for ideals and modules of codimension 3 in \cite{Gue-Wey}, \cite{GNW1}, \cite{Gue-SAF}.

Given a Gorenstein ideal $I$ of codimension 4 in a Noetherian ring $R$, we consider its minimal free resolution $\FF$ adopting the same notation used in (\ref{complexF}). In particular, the basis of the middle module $G$ is assumed to be in the hyperbolic form.
Let $n$ be the minimal number of generators of $I$.

The three critical representations are denoted $W_1$, $W_{n}$, $W_0$ and the linear maps corresponding to the Schur functors in each representation are denoted by $w^{(i)}_{j,k}$, where $i$ stands for the index of the critical representation, $j$ for the graded component and the index $k$ is used to distinguish different maps appearing in the same graded component. However, when there is a unique map in a given graded component, we will omit the index $k$.

After tensoring the critical representations with the ring $R$, we define the higher structure maps corresponding to the zero and first graded components.
The maps \mbox{$w^{(1)}_{0}: F \to G^*=G$} and \mbox{$w^{(n)}_{0}: R \to F$}  correspond to the differentials $d_3$ and $d_4$ of $\FF$, while the third map in the lowest degree \mbox{$w^{(0)}_{0}: V(\omega_{2}, D_{n-1}) \to R$} assigns to each generator of the half-spinor representation $ V(\omega_{2}, D_{n-1})$ the corresponding spinor coordinate in $R$ (see Section \ref{section:pre}).

We identify the generators of the half-spinor representation $V(\omega_{2}, D_{n-1})$ with subsets of \mbox{$\lbrace 1, \ldots, n-1 \rbrace$} of even cardinality (including the empty set). The generators of $V(\omega_{3}, D_{n-1})$ are instead associated to the subsets of even cardinality. For every subset $H \subseteq \lbrace 1, \ldots, n-1 \rbrace$ we denote by $u_H$ the corresponding element of one of the half-spinor representation. If $H=\lbrace  i_1, \ldots, i_t \rbrace$, we sometimes abuse the notation by writing $u_{i_1 \ldots i_t}:=u_{\lbrace i_1, \ldots, i_t \rbrace}$. We also denote by $u_{H}^{*}:=u_J$ the dual element of $u_H$, where $J:= \lbrace 1, \ldots, n-1 \rbrace \setminus H$. 

We start by describing the map corresponding to the first graded component of $W_{n}$. It is defined by a lift of a certain cycle along the third differential. More precisely, the higher structure map $w^{(n)}_1: V(\omega_{3}, D_{n-1}) \to F$ is obtained by lifting the image of the map \mbox{$q^{(n)}_1:V(\omega_{3}, D_{n-1}) \to G$} defined for $u_H$ with $|H|$ odd as
$$ q^{(n)}_1(u_H)= \sum_{i \not \in H} w^{(0)}_0(u_{H \cup \lbrace i \rbrace}) e_i +  \sum_{i \in H} w^{(0)}_0(u_{H \setminus \lbrace i \rbrace}) \hat{e}_i. $$

The composition $d_2 \circ q^{(n)}_1$ is zero. To see this we argue as in the proof of \cite[Theorem 2.1]{Guerrieri-Ni-Weyman}. It is clear that structure maps of any localization of $\FF$ are obtained as localizations of the structure maps of $\FF$. Moreover, the relations defining structure maps are equivariant with respect to any change of basis in $F$ and $G$. 
Take any non-zero divisor $x\in I$, which exists since by assumptions $I$ has codimension 4. Over the localization $R_x$ the complex $\mathbb{F}$ splits and $\operatorname{Spec} R_x$ is dense in $\operatorname{Spec} R$. Thus it is enough to prove the hypothesis when $\mathbb{F}$ is split exact. Consider the split exact complex $\FF^{split}$ of the same format as $\FF$ and with the differentials such that $d_4(1_{F_4})= f_n$, $d_3(f_i)= \hat{e}_i$ for $i \leq n-1$, $d_3(f_n)=0$, 
$d_2(e_i)= f_i^*$, $d_2(\hat{e}_i)= 0$, $d_1(f_i^{*}) = 0$  for $i \leq n-1$, $d_1(f_n^*)=1$. All the spinor coordinates of $\FF^{split}$ are zero, except for $w^{(0)}_0(u_{\emptyset})= 1$.
It follows that on $\FF^{split}$ we have
$$
q^{(n)}_1(u_{i})=\hat{e_i},\quad q^{(n)}_1(u_H)=0\textnormal{ if }|H|>1.
$$
The map $q^{(n)}_1$ is defined in equivariant way with respect to change of basis of $F$ and $G$.
It follows that $q^{(n)}_1(u_H)$ is a cycle also for the original complex $\FF$. The image of $q^{(n)}_1$ can be lifted along the differential $d_3$ to an element of $F$, obtaining the map $w^{(n)}_1: V(\omega_{3}, D_{n-1}) \to F$. Of course, like the multiplicative structure of an exact complex of length 3, this map is not unique. Any two lifts differ by an element of $\ker(d_3)= \operatorname{im}(d_4)$.

Similarly, we can consider the map associated to the first graded component of the critical representation $W_1$. We need a higher structure map $w^{(1)}_1: V(\omega_{2}, D_{n-1}) \otimes F \to R$. Denote the entries of $d_3$ by $y_{ih}= \langle d_3(f_h), e_i^* \rangle$, $\hat{y}_{ih}= \langle d_3(f_h), \hat{e}_i^* \rangle.$
To compute $w^{(1)}_1$ we consider the map $q^{(1)}_1: V(\omega_{2}, D_{n-1}) \otimes F \to F$ defined, for $H$ such that $|H|$ is even, by 
$$ q^{(1)}_1(u_H \otimes f_h) =  \sum_{i \not \in H} \hat{y}_{ih} w^{(n)}_1 (u_{H \cup \lbrace i \rbrace}) +  \sum_{i \in H} y_{ih}  w^{(n)}_1 (u_{H \setminus \lbrace i \rbrace})  - w^{(0)}_0 (u_H) f_h.  $$
Again checking over a split exact complex one can show that image of $q^{(1)}_1$ is a cycle in the complex $\FF$. Therefore it can be lifted to $R$ along the differential $d_4$, obtaining the map $w^{(1)}_1$.

\subsection{The format $E_6$}\label{sec:E6,alpha1}

We look now in detail at the higher structure maps corresponding to the format $E_6$ in the grading induced by the vertex $\alpha_1$.
Resolutions of length 4 and format $E_6$ of Gorenstein ideals correspond to Gorenstein ideals with 6 generators. Thus the minimal free resolution $\mathbb{F}$ is:
\begin{equation}
\label{complexF6}
\FF: 0 \longrightarrow R \buildrel{d_4}\over\longrightarrow R^6 \buildrel{d_3}\over\longrightarrow  R^{10} \buildrel{d_2}\over\longrightarrow R^6 \buildrel{d_1}\over\longrightarrow R.
\end{equation}
In this case the two half-spinor representations of $\mathfrak{so}(10)$ are denoted by $V(\omega_2, D_5)$ and $V(\omega_3, D_5)$.
The three critical representations for this format are: 
$$W_1=F^* \boxtimes [\C \oplus V(\omega_2, D_5) \oplus G ],  $$
$$W_6=F \boxtimes [G \oplus V(\omega_3, D_5) \oplus \C ],$$
$$W_0= \C \boxtimes \Big[V(\omega_2, D_5) \oplus [\C \oplus \mathfrak{so}(10)] \oplus V(\omega_3, D_5)\Big]$$
Note that $W_1\simeq W_6^*$ and $W_0$ is self-dual.

We adopt the same notation as in Section \ref{sec:firstcomponents,alpha1}, i.e. $u_H$ denotes the element of the half-spinor representations corresponding to a subset $H \subseteq \lbrace 1, \ldots, 5 \rbrace$. If $|H|$ is even, $u_H$ is an element of $V(\omega_3, D_5)$, otherwise of $V(\omega_2, D_5)$. The maps $w^{(1)}_0$, $w^{(6)}_0$, $w^{(0)}_0$, $w^{(1)}_1$, $w^{(6)}_1$ can be computed as described in Section \ref{sec:firstcomponents,alpha1}. We will look at explicit computations in the case of a split exact complex and in the case of the free resolution of an hyperplane section of a generic Gorenstein ideal of codimension 3 with 5 generators.
Now we describe how to compute the maps $w^{(1)}_2$, $w^{(6)}_2$ and the higher structure maps associated to $W_0$.

For the map $w^{(1)}_2: G \to F$ we first consider a composition map $$ q^{(1)}_2: G \to V(\omega_2, D_5) \otimes V(\omega_2, D_5) \to F \otimes F $$ defined by
$$ q^{(1)}_2(e_i)= \sum_{j \neq i} (-1)^{i+j-1} w^{(1)}_1(u_{j}) \otimes w^{(1)}_1(u_{ij}^*), \mbox{ and }$$
$$ q^{(1)}_2(\hat{e}_i)=  \sum_{j,k \neq i} (-1)^{i+j+k-1} w^{(1)}_1(u_{ijk}) \otimes w^{(1)}_1(u_{jk}^*) - w^{(1)}_1(u_{i}) \otimes w^{(1)}_1(u_{\emptyset}^*). $$
\noindent The image of $ q^{(1)}_2$ is a cycle in the exact complex $\FF \otimes F$ and can be lifted to $F \cong F \otimes R$ along the map $\operatorname{id}_F \otimes d_4: F \otimes R \to F \otimes F$. 

For $w^{(6)}_2: F \to R$ we consider the map $q^{(6)}_2: F \to F $ defined by 
$$ q^{(6)}_2( f_h) =  \sum_{\footnotesize |H| \mbox{ odd }} (-1)^{\sigma(h,H)} w^{(6)}_1 (u_{H}^* \otimes f_h) \cdot w^{(1)}_1 (u_{H})  + w^{(1)}_2 (d_3(f_h)).  $$ Also in this case the image of this map lifts to $R$ along the differential $d_4$.

In the first and second graded components of the critical representation $W_0$ we have maps $w^{(0)}_{1,1}: R \to R$, $w^{(0)}_{1,2}: \mathfrak{so}(10) \to R$ and $w^{(0)}_{2}: V(\omega_2, D_5) \to R$. The sources of $w^{(0)}_{1,1}$ and $w^{(0)}_{1,2}$ can be identified with appropriate subrepresentations of $V(\omega_2, D_5) \otimes V(\omega_3, D_5)$. The tensor product $V(\omega_2, D_5) \otimes V(\omega_3, D_5)$ is multiplicity-free, so the identification is unique up to a scalar. For the first map we consider the trivial representation and identify the unit element of $R$ with the element
$$ \sum_{\footnotesize |H| \mbox{ odd }} (-1)^{\sigma(H)} u_H \otimes u_H^*. $$ 
For the second map recall that we have an isomorphism $\mathfrak{so}(10)\simeq \bigwedge^2\mathbb{C}^{10}$, which is the irreducible representation for $D_5$ with the highest weight $\omega_2$. The standard basis is given by $\theta_{ij}:=e_i\wedge e_j,\ \theta_{\widehat{i} \widehat{j}}:=\hat{e_i}\wedge\hat{e_j}$ for $1 \leq i < j \leq 5$ and $\theta_{i \widehat{j}}:=e_i\wedge\hat{e_j}$ for $1 \leq i,j \leq 5$. Of course $\theta_{12}$ is the highest weight vector. Computing the unique highest weight vector of $V(\omega_2,D_5)\otimes V(\omega_3,D_5)$ of weight $\omega_2$, we conclude that the second map must send
$$  \theta_{12}\mapsto \displaystyle\sum_{\substack{H\subset\{3,4,5\}, \\ \textnormal{|H| even}}}(-1)^{\sum_{j\in H}j}\cdot u_{\{1,2\}\cup H}\otimes u_{\{1,2\}\cup H^c}=$$
$$=u_{12}\otimes u_{12345}-u_{1234}\otimes u_{125}+u_{1235}\otimes u_{124}-u_{1245}\otimes u_{123}.$$
The entire map is then determined by $\mathfrak{so}(10)$-equivariance. For example if $g\in\mathfrak{so}(10)$ maps $e_1\mapsto \hat{e_2}$ and $e_2\mapsto\hat{e_1}$, we have $g.\theta_{12}=-\theta_{1\widehat{1}}-\theta_{2\widehat{2}}$ and consequently
$$
\theta_{1\widehat{1}}+\theta_{2\widehat{2}}\mapsto -u_\varnothing\otimes u_{12345}-u_{12}\otimes u_{345}+u_{34}\otimes u_{125}+u_{1234}\otimes u_{5}-u_{35}\otimes u_{124}$$
$$-u_{1235}\otimes u_{4}+u_{45}\otimes u_{123}+u_{1245}\otimes u_{3}.
$$
After this identification, the higher structure maps $w^{(1)}_2$ and $w^{(6)}_2$ are computed by composing the inclusions $\iota_1:R\hookrightarrow V(\omega_2, D_5) \otimes V(\omega_3, D_5)$ and $\iota_2:\mathfrak{so}(10)\hookrightarrow V(\omega_2, D_5) \otimes V(\omega_3, D_5)$ with the map
$$q^{(0)}_{1} = w^{(1)}_{1} \otimes w^{(0)}_{0}: V(\omega_2, D_5) \otimes V(\omega_3, D_5) \to F \otimes R $$
and then lifting to $R$ along the differential $d_4$ of $\FF$.

The last map $w^{(0)}_{2}$ is computed by lifting the image of the map $q^{(0)}_{2}: V(\omega_2, D_5) \to F $ defined by 
$$ q^{(0)}_{2}(u_H)= \sum_{i \not \in H} w^{(1)}_2(e_i) w^{(0)}_0(u_{H \cup \lbrace i \rbrace}) +  \sum_{i \in H} w^{(1)}_2(\hat{e}_i) w^{(0)}_0(u_{H \setminus \lbrace i \rbrace}) - w^{(1)}_1(u_i) w^{(0)}_{1,1}(1)$$
Since the critical representations all have three graded components, this gives the last of the higher structure maps.

\subsection{Split exact complex of format $E_6$}\label{sec:split,alpha1}

In the proof of our main structure theorem, the most important complex that we need is the most trivial one: the split exact complex. In general we observed that higher structure maps are lifts of certain cycles and thus not unique. This non-uniqueness can be accounted for by extending the ring $R$ with new independent variable, called \emph{defect variables}, and using them to parametrize choices of different lifts.
The computation of higher structure maps with generic defect variables for a split exact complex of formats $E_6$ 
is the starting point for the proof of Lemma \ref{lem:square-matrix}.

Let us then compute higher structure maps for the split exact complex of format $E_6$, i.e. with ranks $(1,6,10,6,1)$. %
In this case we compute each map choosing generic liftings, parametrized by the defect variables. 
In the next subsection we perform the same computation, but without introducing new variables, for the minimal free resolution of a hyperplane section of a Gorenstein ideal of codimension 3 with 5 generators. 
We show that these two complexes are dual to each other in the following sense: the matrices of the top components of the critical representations of one complex give exactly the lowest components, i.e. differentials and spinor coordinates, of the other. This duality points directly to the fact that hyperplane sections of Gorenstein ideals of codimension 3 with 5 generators are the generic model for Gorenstein ideals of codimension 4 with 6 generators.
We prove this result in Theorem \ref{th:generic-gorenstein}.

Let $\FF$ be a split exact complex of format $(1,6,10,6,1)$.
The differentials of $\FF$ can be described by setting $d_4(1)= f_6$, $d_3(f_i)= \hat{e}_i$ for $1\leq i \leq 5$, $d_3(f_6)=0$, 
$d_2(e_i)= f_i^*$, $d_2(\hat{e}_i)= 0$, $d_1(f_i^{*}) = 0$  for $1\leq i \leq 5$, $d_1(f_6^*)=1$. The spinor coordinates of $\FF$ are all zero, except $w^{(0)}_0(u_{\emptyset})= 1$.

First we compute the higher structure map $w^{(1)}_1$ on $\FF$ with generic liftings, introducing new variables $b_H$ for every $H \subseteq \lbrace 1,2,3,4,5 \rbrace$ such that $|H|$ is odd. Using the definition of $w^{(1)}_1$, we immediately get:
$$ w^{(1)}_1(u_{i})= f_i + b_{i}f_6 \quad \mbox{ and } \quad w^{(1)}_1(u_H)= b_Hf_6 \quad \mbox{ if } \quad |H| > 1.    $$
To compute $w^{(1)}_2$ we notice that 
$$ q^{(1)}_2(e_1)= (f_2 + b_{2}f_6) \otimes b_{345}f_6 - (f_3 + b_{3}f_6) \otimes b_{245}f_6+ (f_4 + b_{4}f_6) \otimes b_{235}f_6- (f_5 + b_{5}f_6) \otimes b_{234}f_6, $$
$$ q^{(1)}_2(\hat{e}_1)=  b_{123}f_6 \otimes b_{145}f_6- b_{124}f_6 \otimes b_{135}f_6+  b_{125}f_6 \otimes b_{134}f_6- (f_1 + b_{1}f_6) \otimes b_{12345}f_6. $$
Thus 
$$ w^{(1)}_2(e_1)=  b_{345}f_2 -  b_{245}f_3+  b_{235}f_4- b_{234}f_5 + (b_{2}b_{345}-b_{3}b_{245} + b_{4}b_{235} - b_{5}b_{234}) f_6, $$
$$ w^{(1)}_2(\hat{e}_1)= -b_{12345}f_1 + (P_1-b_{1}b_{12345}) f_6, $$ where the pfaffian $P_1$ is defined as
$P_1:= b_{123} b_{145}- b_{124} b_{135}+  b_{125}b_{134}$.
The other entries are obtained by permutation of the indices.

For the map $w^{(6)}_1$ we recall that $w^{(0)}_0(u_H)=0$ if $H \neq \emptyset$ and the entries of $d_3$ are $\hat{y}_{ii} =1$, $\hat{y}_{ih} = 0$ if $i \neq h$, $y_{ih}=0$.
Hence, 
$$ w^{(6)}_1(u_{\emptyset} \otimes f_h)= b_h \, \mbox{ for } 1\leq h \leq 5, \quad w^{(6)}_1(u_{\emptyset} \otimes f_6)= -1, $$ 
$$ w^{(6)}_1(u_{12} \otimes f_h)= 0 \, \mbox{ for } h =1,2,6, \quad w^{(6)}_1(u_{12} \otimes f_h)= b_{12h} \, \mbox{ for } h =3,4,5, $$ 
$$ w^{(6)}_1(u_{1234} \otimes f_h)= 0 \, \mbox{ for } h \neq 5, \quad w^{(6)}_1(u_{1234} \otimes f_5)= b_{12345}. $$ 
To compute the next map we use the fact that for $1\leq h \leq 5$, $w^{(6)}_1(u_{H} \otimes f_h)$ is nonzero if and only if $h \not \in H$. This gives
$$  q^{(6)}_2(f_1)= b_1b_{12345}f_6 + b_{12345}(f_1 + b_1f_6) + P_1f_6 + w^{(1)}_2(\hat{e}_1) = (b_1b_{12345}+2P_1)f_6,$$
and similarily for other indices $\neq 6$. Thus we obtain
$$ w^{(6)}_2(f_h)= b_1b_{12345}+2P_h \quad \textnormal{for}\quad 1\leq h \leq 5\quad \textnormal{and}\quad w^{(6)}_2(f_6)= -b_{12345}.$$

The ideal generated by the entries of $w^{(6)}_2$ is $(P_1, \ldots, P_5, b_{12345})$. It is a hyperplane section of the ideal of submaximal pfaffians of the skew-symmetric matrix defined by the variables $b_{ijk}$. The matrix of $w^{(1)}_2$ is a presentation matrix of this ideal. On the other hand, the top component of the critical representation $W_0$ contains information about the spinor structure of this related ideal.

The computation of $w^{(0)}_{1,1}$ is straightforward. Using the formulas for $w^{(0)}_0$ we simply get $w^{(0)}_{1,1}(1)= b_{12345}$. Also for $w^{(0)}_{1,2}$ we easily obtain $w^{(0)}_{1,2}(\theta_{ij}) = w^{(0)}_{1,2}(\theta_{\widehat{i}j}) = w^{(0)}_{1,2}(\theta_{i \widehat{j}}) = 0$ and $w^{(0)}_{1,2}(\theta_{\widehat{i} \widehat{j}}) = b_{H}$ where $H = \lbrace 1,2,3,4,5 \rbrace \setminus \lbrace i,j \rbrace$. For $w^{(0)}_{2}$, if $|H| =3, 5$ we get $q^{(0)}_{2}(u_h) = b_{12345} \cdot w^{(1)}_{1}(u_I)$ and $w^{(0)}_{2}(u_H) = b_H b_{12345}$. For $H=\lbrace i \rbrace\subset\{1,\cdots,5\}$, we get
$$q^{(0)}_{2}(u_i) = w^{(1)}_2(\hat{e}_i)+ (f_i + b_i f_6) b_{12345} = P_if_6$$
and consequently $w^{(0)}_{2}(u_i) = P_i$. We will shortly see that the entries of $w^{(0)}_2$ give a choice of spinor coordinates, i.e. a choice of the structure map $w^{(0)}_0$, for the minimal free resolution of the ideal $(P_1, \ldots, P_5, b_{12345})$.

\subsection{The minimal free resolution of a hyperplane section of pfaffians}\label{sec:pfaffians,alpha1}

Let us now analyze higher structure maps in the case of the complex $\FF$ which is the minimal free resolution of a hyperplane section of the ideal generated by submaximal pfaffians of a $5 \times 5$ skew-symmetric matrix. Define the skew-symmetric matrix on variables $x_{ij}=-x_{ji}$ and denote its pfaffians using the notation
$$ P_{\widehat{1}}= x_{23}x_{45} - x_{24} x_{35} + x_{25} x_{34}, $$ with the other pfaffians obtained from this one by permutation of indices. Let $y$ be an element of $R$ regular modulo the ideal $(P_{\widehat{1}}, \ldots, P_{\widehat{5}})$. The complex $\FF$ resolving the Gorenstein ideal $(P_{\widehat{1}}, \ldots, P_{\widehat{5}}, y)$ has format $(1,6,10,6,1)$. We consider bases on $F$ and $G$ such that the differentials of $\FF$ are
$$ \footnotesize d_4 =\bmatrix P_{\widehat{1}}  \\ P_{\widehat{2}} \\ P_{\widehat{3}} \\ P_{\widehat{4}} \\ P_{\widehat{5}} \\ y \endbmatrix , \quad 
d_3=\bmatrix -y & 0 & 0 & 0 & 0 & P_{\widehat{1}} \\  0 & -y & 0 & 0 & 0 & P_{\widehat{2}} \\ 0 & 0 & y & 0 & 0 & -P_{\widehat{3}} \\ 0 & 0 & 0 & -y & 0 & P_{\widehat{4}} \\ 0 & 0 & 0 & 0 & -y & P_{\widehat{5}} \\ 
0 & x_{12} & -x_{13} & x_{14} & -x_{15} & 0 \\
 -x_{12} & 0 & x_{23} & -x_{24} & x_{25} & 0 \\
 -x_{13} & x_{23} & 0  & -x_{34} & x_{35} & 0 \\
x_{14} & -x_{24} & x_{34} & 0 & -x_{45} & 0 \\
 -x_{15} & x_{25} & -x_{35} & x_{45} & 0 & 0 \\
\endbmatrix,\quad
$$
This choice makes the basis of $G$ hyperbolic with respect to the bilinear symmetric form $Q:G\otimes G\rightarrow R$ induced by the multiplication $\FF_2\otimes \FF_2\rightarrow \FF_4$. 
The spinor coordinates of $\FF$ are $$ w^{(0)}_0(u_{\emptyset})=y^2, \quad w^{(0)}_0(u_{ij})= \pm yx_{ij}, \quad w^{(0)}_0(u_{i}^*)= P_{\widehat{i}}.$$

We use the spinor coordinates to compute the other higher structure maps. For $w^{(1)}_1(u_{1})$ we write the spinor cycle
$q^{(1)}_1(u_{1})= u_{12} e_2 + u_{13} e_3+ u_{14} e_4 + u_{15} e_5 + u_{\emptyset} \hat{e}_1. $ Looking at  the columns of $d_3$ and lifting we get $w^{(1)}_1(u_{1}) = yf_1 $ and analogously
$$ w^{(1)}_1(u_{i}) = yf_i\quad\textnormal{for}\quad 1\leq i\leq 5. $$
Similarly we write $q^{(1)}_1(u_{123})= u_{1234} e_4 + u_{1235} e_5+ u_{23} \hat{e}_1 + u_{13} \hat{e}_2 + u_{12} \hat{e}_3.$ Lifting this kind of expressions we get the general formula
$$  w^{(1)}_1(u_{ijk}) = x_{ij} f_k - x_{ik} f_j + x_{jk} f_i\quad\textnormal{for}\quad 1\leq i<j<k\leq 5.$$
Finally, $q^{(1)}_1(u_{12345})= \sum_{i=1}^5 u_i^* \hat{e}_i $ and therefore $$ w^{(1)}_1(u_{12345}) = f_6. $$
To compute the map $w^{(1)}_2$ we apply the general formula to obtain 
$$ q^{(1)}_2(e_1)= yf_2 \otimes (x_{34} f_5 - x_{35} f_4 + x_{45} f_3) - yf_3 \otimes (x_{24} f_5 - x_{25} f_4 + x_{45} f_2)+  $$
$$ +yf_4 \otimes (x_{23} f_5 - x_{25} f_3 + x_{35} f_2)-yf_5 \otimes (x_{23} f_4 - x_{24} f_3 + x_{34} f_2) = 0;     $$
$$ q^{(1)}_2(\hat{e}_1)= yf_1 \otimes f_6 + (x_{12} f_3 - x_{13} f_2 + x_{23} f_1) \otimes (x_{14} f_5 - x_{15} f_4 + x_{45} f_1)+   $$
$$ - (x_{12} f_4 - x_{14} f_2 + x_{24} f_1) \otimes (x_{13} f_5 - x_{15} f_3 + x_{35} f_1) +(x_{12} f_5 - x_{15} f_2 + x_{25} f_1) \otimes (x_{13} f_4 - x_{14} f_3 + x_{34} f_1) = $$ $$ = f_1 \otimes d_4(1_{F_4}).$$
Hence, 
$$ w^{(1)}_2(e_i)=0, \quad w^{(1)}_2(\hat{e}_i)= f_i\quad\textnormal{for}\quad 1\leq i\leq 5.   $$

We move to the maps associated to the critical representation $W_1$. An easy computation shows that the only nonzero entries of $w^{(6)}_1$ are
$$  w^{(6)}_1(u_{\emptyset} \otimes f_6)=y, \quad w^{(6)}_1(u_{ij} \otimes f_6)= y x_{ij}, \quad  w^{(6)}_1(u_{i}^* \otimes f_i)= 1. $$ Using the previous maps and the general formulas we further obtain 
$$ w^{(6)}_2(f_i)= 0\quad \mbox{ for } 1\leq i\leq 5, \quad w^{(6)}_2(f_6)= 1. $$
Observe that the maps $w^{(1)}_2$ and $w^{(6)}_2$ have the same matrices as the third and fourth differentials of the split exact complex that we considered in Section \ref{sec:split,alpha1}.

Finally we look at the maps in $W_0$. For $w^{(0)}_{1,1} $ we have to lift an expression involving all the terms of the form $\pm y x_{12} (x_{34}f_5 - x_{35}f_4 + x_{45}f_3)$ and a term of the form $y^2 f_6+ \sum_{i=1}^5 P_{\hat{i}} yf_i $. The choice of signs is such that the first part cancels out and the second part lifts to $w^{(0)}_{1,1}(1) = y$. We use this map to compute $w_2^{(0)}$ and show that the entries of this map coincide with the spinor coordinates of the split exact complex from the previous section, after applying the isomorphism $u_H \to u^*_H$. Indeed, for $H$ such that $|H|=1,3$, it is easy to check that $q_2^{(0)}(u_H)=0$. For $ H=\lbrace 1,2,3,4,5 \rbrace$, we get $q_2^{(0)}(u_H)= d_4(1_{D_4})$ and therefore $w_2^{(0)}(u_H)=1$.

\section{Higher structure maps from the grading induced by $\alpha_2$}\label{section:hsm2}

In this section we consider the higher structure maps associated to the grading induced by the node $\alpha_2$ of the diagram pictured in Example \ref{ex:2-grading}.

\subsection{First graded components}\label{sec:firstcomponents,alpha2}

To prove our main Theorem \ref{th:generic-gorenstein}, we need another set of higher structure maps. This time they are related to the $\alpha_2$-grading on the Lie algebra $\mathfrak{g}(E_6)$ and their domains decompose into irreducible representations of $\mathfrak{sl}(F)\simeq\mathfrak{g}(A_{6})$, which is the Lie algebra with the Dynkin diagram $E_6\setminus\{\alpha_2\}$. These higher structure maps are still induced by graded components of critical representations. However, unlike for the $\alpha_1$-grading, we do not have a generic object associated with them. Nevertheless, they are crucial in the proof of our main theorem. To prove their existence we define them by lifting cycles.

While for the structure maps associated with the simple root $\alpha_1$ the spinor structure is the main point of focus, here the starting point is the multiplicative structure on the complex $\mathbb{F}$. There there are two relevant critical representations: $W_1:=V(\omega_1,E_n)\otimes\mathbb{C}$ and $W_6:=V(\omega_6,E_6)\otimes G$. To study them we adopt the more classical notation $F_1, F_2, F_3, F_4$ to denote the free modules of the complex $\FF$.
For the format $E_6$ the critical representations decompose into irreducibles as
$$
W_1=[F_1 \oplus\bigwedge^4 F_1 \boxtimes F_4^* \oplus \bigwedge^6F_1 \otimes F_1 \boxtimes S_2 F_4^*] \boxtimes \mathbb{C}.
$$
$$
W_6=[F_1^* \oplus \bigwedge^2 F_1 \oplus \bigwedge^5 F_1 ]\boxtimes F_2^*.
$$

The maps in degree $0$ are the same as for the $\alpha_1$-grading described in the previous section. Indeed, the degree-zero components are just the $(0,0)$-part of the bi-grading induced by the pair of roots $(\alpha_1,\alpha_2)$ on the Lie algebra $\mathfrak{g}(E_6)$. Thus, after tensoring the critical representations with the ring $R$, the maps $w_0^{(6)}:F_1^* \cong F_3 \rightarrow F_2$ and $w^{(1)}_0:F_1 \rightarrow R$ are the differentials.

The maps in degree one can be described for arbitrary ranks of $F_1$ in terms of the multiplicative structure. This time we denote by $f_1, \ldots, f_n$ the basis of $F_1$ and again we use the notation with $e_i, \hat{e}_i$ for the basis of $F_2=G$. 

For the critical representation $W_6$ we have the linear map $w^{(6)}_1:\bigwedge^2 F_1 \rightarrow F_2$ given by the standard skew-symmetric product defined by the comparison of $\FF$ with the Koszul complex on the generators of Im$(d_1)$.

Thus 
$w^{(6)}_1(f_i\wedge f_j)=f_i\cdot f_j\in F_2$. Similarly the structure map $w^{(1)}_1:\bigwedge^4 F_1 \rightarrow R$ is determined by the usual multiplicative structure, induced by the same comparison with the Koszul complex (see Section \ref{section:pre1}). 
In the following, to simplify the notation, we replace the images of maps $_1$ and $w^{(1)}_1$ by products $f_i \cdot f_j$ and $f_{i_1} \cdot f_{i_2} \cdot f_{i_3} \cdot f_{i_4} $, and we also use the same product notation for the map 
$F_2 \otimes F_2 \rightarrow R$.

The structure map in the second graded component of $W_6$ is $w^{(6)}_2:\bigwedge^5 F_1 \rightarrow F_2$. We write down the formula for the choice of basis elements $f_1 \wedge \ldots \wedge f_5 $. For the other elements it can be obtained by permutation with the usual rules for changing the signs.
The element
\begin{equation}
\label{q3,2}
q^{(6)}_2(f_1\wedge\ldots\wedge f_5)=\displaystyle\sum_{j=1}^{5} (-1)^{j+1} w^{(1)}_1( f_1\wedge \ldots \widehat{f_j} \ldots \wedge f_5) \cdot f_j \in F_1
\end{equation}
defines a cycle. The map $w^{(6)}_2$ is a lift of the cycle $q^{(6)}_2$ along the differential $d_2: F_2 \rightarrow F_1$.

The higher structure maps in the critical representation $W_1$ are obtained using more complicated formulas lifting cycles in some complex. For the purpose of this paper we will use only the first of this maps in the case of $E_6$ format. Such a map $w^{(1)}_2: \bigwedge^6 F_1 \otimes F_1 \rightarrow S_2F_4^* \cong R$ is defined by the composition  
\begin{equation}
\label{q1,2}
\bigwedge^6 F_1 \otimes F_1 \rightarrow \bigwedge^5 F_1 \otimes \bigwedge^2 F_1 \rightarrow F_2 \otimes F_2 \rightarrow F_4 \cong R,
\end{equation}
 where the first arrow is the map 
$$f_1 \wedge \ldots \wedge f_6 \otimes f_1 \to \sum_{i=1}^6 (-1)^{i} (f_1 \wedge \ldots \widehat{f_i} \ldots \wedge f_6) \otimes (f_1 \wedge f_i),$$ the second arrow is $w^{(6)}_2 \otimes w^{(6)}_1$, and the last arrow is the multiplication $F_2 \otimes F_2 \rightarrow R$. 

\subsection{Split exact complex
 of format $E_6$}\label{sec:split,alpha2}

Analogously as we did in Section \ref{section:hsm1} for the $\alpha_1$-grading, we compute the higher structure maps in the $\alpha_2$-grading for a split exact complex of the format $E_6$, using defect variables to parametrize generic choices of liftings.

Let $\mathbb{F}$ be the split exact complex of format $E_6$ with $d_4(1_{F_4})= f_6^*$, $d_3( f_6^*)= 0$, $d_1(f_6)= 1$, and for $k=1, \ldots, 5$,
$d_3(f^*_k)= \hat{e}_k$, $d_2(e_k)= f_k$, $d_2(\hat{e}_k)= 0 = d_1(f_k)$. The maps $w^{(6)}_1,\ w^{(1)}_1$ are determined by a multiplicative structure on $\mathbb{F}$. Let us fix again a basis $e_1,\cdots,e_{5}, \hat{e_1},\cdots,\hat{e}_{5}$ of $G$ so that $e_i\cdot e_j=\hat{e_i}\cdot\hat{e_j}=0$ and $e_j\cdot\hat{e_j}=\delta_{i,j}$. Furthermore, we may assume
$$\begin{cases}
f_i\cdot f_j=0\quad\textnormal{for}\quad 1\leq i<j\leq 5\\
f_i\cdot f_6
 =e_i\quad\textnormal{for}\quad 1\leq i\leq 5\\
\end{cases}$$

Note that the multiplication map $w^{(6)}_1:\bigwedge^2 F_1 \rightarrow F_2$ is a lift of the map
$$q^{(6)}_1(f_i\wedge f_j)=\begin{cases}
0\quad\textnormal{for}\quad i,j\neq 6\\
f_i\quad\textnormal{for}\quad j=6
\end{cases}$$
Hence the generic lift of $q^{(6)}_1$ along the differential $d_2$ is
$$w^{(6)}_1(f_i\wedge f_j)=\begin{cases}
0+\sum b^{k}_{i,j}\hat{e_k}\quad\textnormal{for}\quad i,j\neq 6\\
e_i+\sum b^{k}_{i,n}\hat{e_k}\quad\textnormal{for}\quad j=6, \end{cases}$$
where $i,j \in  \lbrace 1,\ldots,6 \rbrace$, $k=1,\ldots, 5$ and $b^k_{i,j}=-b^k_{j,i}$ are defect variables. Thus we initially need to introduce $\binom{6}{2}\cdot 5=75
$ new variables and add them to the base ring $R$.

However, we still want the bilinear form on $F_2$, which is induced by the multiplication $w^{(6)}_1$, to be in the hyperbolic form. Thus for $1\leq i<j\leq 5$ we must impose the conditions
$$0= f_i \cdot f_6 \cdot f_j \cdot f_6 =  w^{(1)}_1(f_i\wedge f_6)\cdot w^{(1)}_1(f_j\wedge f_6)=\left(e_i+\sum b^{k}_{i,6}\hat{e_k}\right)\cdot\left(e_j+\sum b^{k}_{j,6}\hat{e_k}\right)=b^j_{i,6}+b^i_{j,6},$$
$$0= f_
i \cdot f_j \cdot f_i \cdot f_6 = w^{(1)}_1(f_i\wedge f_j)\cdot w^{(1)}_1(f_i\wedge f_6)=\left(\sum b^{k}_{i,j}\hat{e_k}\right)\cdot\left(e_i+\sum b^{k}_{i,6}\hat{e_k}\right)=b^i_{i,j}.$$
It follows that for $1\leq i<j<k\leq 5$ and any permutation $\sigma:\{i,j,k\}\rightarrow\{i,j,k\}$ the defect variables satisfy $b^i_{j,k}=\operatorname{sgn}(\sigma)\cdot b^{\sigma(i)}_{\sigma(j),\sigma(k)}$, and furthermore $b^{i}_{j,6}=-b^{j}_{i,6}$. Thus the defect variables $b^{i}_{j,k}$ with $1\leq i<j<k\leq 6$ generate all the others. Consequently the number of independent defect variables in this first set equals $\binom{n-1}{2}+\binom{n-1}{3}=\binom{6}{3}=20$.

We move on to the map $w^{(1)}_1$. By associativity of the multiplication, this map can be computed in a compatible way with our choice of $w^{(6)}_1$ as
$$w^{(1)}_1(f_i\wedge f_j\wedge f_k\wedge f_l)=w^{(6)}_1(f_i\wedge f_j)\cdot w^{(6)}_1(f_k\wedge f_l)=\begin{cases}
0\quad\textnormal{for}\quad i,j,k,l\neq 6\\
b^{k}_{i,j}\quad\textnormal{for}\quad l=6.
\end{cases}$$

We consider now the top higher structure map in $W_6$: $w^{(6)}_2:\bigwedge^5 F_1 \rightarrow F_2$. 
This map is obtained as a lift of the cycle $q^{(6)}_2(f_{i_1} \wedge \ldots \wedge f_{i_5} )$, defined in equation (\ref{q3,2}). We have 
$$q^{(6)}_2(f_{i_1} \wedge \ldots \wedge f_{i_5} )=\begin{cases}
 b^{i_3}_{i_1,i_2} f_{i_4} - b^{i_4}_{i_1,i_2} f_{i_3} + b^{i_4}_{i_1,i_3} f_{i_2} - b^{i_4}_{i_2,i_3} f_{i_1} \quad\textnormal{for}\quad i_5 = 6\\
0\quad\quad\quad\quad\textnormal{for}\quad i_1,\ldots, i_5 \neq 6
\end{cases}$$
We introduce a second set of defect variables $c^k_{i_1, \ldots, i_5}$ for distinct $i_1, \ldots, i_5\in\{
1,\cdots,6\}$, and $k=1,\ldots,5$. In this way we have
$$w^{(6)}_2(f_{i_1} \wedge \ldots \wedge f_{i_5})=\begin{cases}
\sum c^k_{i_1, \ldots, i_5}\hat{e_k} + b^{i_3}_{i_1,i_2} e_{i_4} - b^{i_4}_{i_1,i_2} e_{i_3} + b^{i_4}_{i_1,i_3} e_{i_2} - b^{i_4}_{i_2,i_3} e_{i_1} \quad\textnormal{for}\quad i_5 = 6\\
\sum c^k_{i_1, \ldots, i_5} \hat{e_k} \quad\quad\quad\quad\textnormal{for}\quad i_1,\ldots, i_5 \neq 6
\end{cases}$$


The defect variables $c^k_{i_1, \ldots, i_5}$ are not independent. In fact, there are relations of dependence between variables $c^k_{i_1, \ldots, i_5}$ and $b^k_{ij}$. 

The map $w^{(6)}_2$ 
satisfies a differential relation \cite[page 10, relation (2)]{palmer} proved by S. Palmer Slattery. 
Let $x_i\in F_1$, $i=1,\cdots,4$, and $v_j\in F_2$, $j=1,2$. Then we have the relation
\begin{dmath*}
\label{diffP1}
v_1\cdot w^{(6)}_2(d_2(v_2)\wedge x_1\wedge \ldots \wedge x_4)+v_2\cdot w^{(6)}_2(d_2(v_1)\wedge x_1\wedge \ldots \wedge x_4)= $$ $$ [x_1\cdot x_2 \cdot x_3 \cdot x_4]\cdot [v_1 \cdot v_2]-\displaystyle\sum_{\sigma\in S_4} \operatorname{sgn}(\sigma)\cdot [v_1 \cdot x_{\sigma(1)} \cdot x_{\sigma(2)}] \cdot [v_2 \cdot x_{\sigma(3)} \cdot x_{\sigma(4)}].
\end{dmath*}

We investigate this relation for the split exact complex case by case, assuming $x_i\in\{f_1,\cdots,f_n\}$ and $v_i\in\{e_1,\cdots,e_{n-1},\hat{e_1},\cdots,\hat{e}_{n-1}\}$. Denote by $RHS$, resp. $LHS$, the right-hand, resp. the left-hand, side of the relation.

The $RHS$ depends only on the variables $b_{ij}^k$.
Since $d_2(\hat{e}_k)=0$, we can assume $v_1 = e_i$, otherwise $LHS = 0$. If $v_2 \in \{\hat{e_1},\cdots,\hat{e}_{n-1}\}$, then also $LHS $ depends only on the variables $b_{ij}^k$. Thus assume $v_2 = e_j$. Set $x_k = f_{i_k}$. If $i=j$ and we choose $x_1, \ldots, x_4 \neq f_i$, we obtain $LHS = 2c_{i, i_1, \ldots, i_4}^{i}$. In particular all the variables $c_{i, i_1, \ldots, i_4}^{i}$ are expressed in terms of the defect variables in the first set. 

If $i = i_1$, $j \neq i_1, \ldots, i_4$, we again get that $LHS = c_{j, i_1, \ldots, i_4}^{i_1}$ is expressed in terms of the $b_{ij}^k$.
But, if $i \neq j$ and both are not in the set $ i_1, \ldots, i_4 $, we get $ LHS = c_{j, i_1, \ldots, i_4}^{i} + c_{i, i_1, \ldots, i_4}^{j} $.

Now, since $n= 6$, we rename the variable $c_{i_1, \ldots, i_5}^{k}$ as $c_{i}^{k}$ where $i$ is the only index different from $i_1, \ldots, i_5$. The above relations shows that $c^{k}_i, c^i_i + c_j^j \in R[b_{ij}^k]$. In particular it is sufficient to introduce only one new defect variable of the form $c^{1}_1$ and
the total number of defect variables is $20+1=21$.


We conclude with the computation of the remaining structure maps in the case of the format $E_6$.

We look back at the relations expressing the variables $c^{k}_i$ in terms of the $ b_{ij}^k $.
First assume that $x_i\neq f_6$ for $i=1,\cdots,4$. In this case both sides of the relation are $0$, unless $v_1=e_i$, $v_2=e_5$, $1\leq i\leq5$. Simple computation shows that if $i<5$, then $RHS=c^i_6$, while if $i=5$, then $RHS=2c^5_6$. To compute the left-hand side, 
 define $$\operatorname{pf}_{1,\hat{6}}:= b^{1}_{2,3}b^{1}_{4,5} - b^{1}_{2,4}b^{1}_{3,5} + b^{1}_{2,5}b^{1}_{3,4},$$ and similarly define $\operatorname{pf}_{i,\hat{j}}$ for $i,j \in \{1,\ldots,6\}$ by the usual permutation rules.
Then we have $LHS=\operatorname{pf}_{i,\hat{6}}$ if $i<5$ and $LHS=2\operatorname{pf}_{i,\hat{6}}$ if $i=5$. 
The relations coming from equation (\ref{diffP1}) show that $c^i_j=\operatorname{pf}_{i,\hat{j}}$ for $i=1,\ldots,6$, $j=1,\ldots,5$, $i\neq j$, and
$c^i_6=\operatorname{pf}_{i,\hat{6}}$ for $i=1,\cdots,5$,
and
Similar computations can be done for the case when one of the $x_i$'s is equal to $f_6$. To get non-zero relation, we must again assume that $v_1,v_2$ are in the isotropic subspace $\langle e_1,\cdots,e_5\rangle$. Considering the case $v_1=v_2$, we obtain $c^i_j=\operatorname{pf}_{i,\hat{j}}$ for $i=1,\cdots,6$, $j=1,\cdots,5$, $i\neq j$. Furthermore, considering the relation for $x_1\wedge\cdots\wedge x_4=f_{i_1}\wedge f_{i_2}\wedge f_{i_3}\wedge f_6$ and $v_1=e_{i}$, $v_1=e_{j}$ such that $\{i_1,i_2,i_3,6\}\cup\{i,j\}=\{1,\cdots,6\}$, we obtain $c^i_i+c^j_j=\operatorname{pf}(i,j)\in R[b^i_{j,k}]$. Here
$$
\operatorname{pf}(1,2)=b^1_{3,4}b^2_{5,6}-b^1_{3,5}b^2_{4,6}+b^1_{3,6}b^2_{4,5}+b^1_{4,5}b^2_{3,6}-b^1_{4,6}b^2_{3,5}+b^1_{5,6}b^2_{3,4}
$$
and $\operatorname{pf}(i,j)$ is defined by the permutation of indices with usual rules for signs.

Using equation (\ref{q1,2}), the last map $w^{(1)}_2: \bigwedge^6 F_1 \otimes F_1 \cong F_1 \rightarrow R$ it is given by
$$w^{(1)}_2(f_i)=\begin{cases}
\operatorname{pf}_{i,\hat{6}}\quad\quad\quad \quad \quad \quad \quad \quad \quad \quad \quad \quad \quad \quad \quad \quad \quad  \textnormal{if}\quad i=1,\cdots,5\\
c^1_1+(\textnormal{expressions in pfaffians involving the index 6}) \quad\quad \quad \quad\textnormal{otherwise}\\
\end{cases}$$

Note that after specializing all the variables $b^{i}_{j,6}\mapsto 0$, $1\leq i<j\leq 5$, the matrices of the top structure maps are
$$ \footnotesize w^{(1)}_2 =\bmatrix \operatorname{pf}_{1}  \\ \operatorname{pf}_{2} \\ \operatorname{pf}_{3} \\ \operatorname{pf}_{4} \\ \operatorname{pf}_{5} \\ c^1_1 \endbmatrix , \quad 
w^{(6)}_2=\bmatrix -c^1_1 & 0 & 0 & 0 & 0 & \operatorname{pf}_{1} \\  0 & -c^1_1 & 0 & 0 & 0 & \operatorname{pf}_{2} \\ 0 & 0 & -c^1_1 & 0 & 0 & \operatorname{pf}_{3} \\ 0 & 0 & 0 & -c^1_1 & 0 & \operatorname{pf}_{4} \\ 0 & 0 & 0 & 0 & -c^1_1 & \operatorname{pf}_{5} \\ 
0 & x_{12} & -x_{13} & x_{14} & -x_{15} & 0 \\
 -x_{12} & 0 & x_{23} & -x_{24} & x_{25} & 0 \\
 x_{13} & -x_{23} & 0  & x_{34} & -x_{35} & 0 \\
-x_{14} & x_{24} & -x_{34} & 0 & x_{45} & 0 \\
 x_{15} & -x_{25} & x_{35} & -x_{45} & 0 & 0 \\
\endbmatrix,
 $$
where $\operatorname{pf}_{i} = \operatorname{pf}_{i, \hat{6}}$.  
Thus these top structure maps are differentials of the minimal free resolution of an hyperplane section of the ideal of pfaffians as we saw for the grading induced by $\alpha_1$.

\subsection{The free resolution of a hyperplane section of pfaffians}\label{sec:pfaffians,alpha2}

To complete the picture we compute higher structure maps for the $\alpha_2$-grading for the hyperplane section of a pfaffians.
As we already performed this type of calculation, we simply write down the results.

The multiplication on $\mathbb{F}$ is easy to compute and the higher structure maps are as follows. In the first degree we have
\begin{dmath*}w^{(6)}_1(f_i^*\wedge f_j^*)=
\begin{cases}
    \hat{e_i}\quad\quad \quad \quad \quad  \quad \quad \quad \quad\quad \quad\textnormal{if}\quad j=6\\
    \pm x_{ab}e_c\pm x_{ac}e_b\pm x_{bc}e_a\quad \textnormal{if}\quad j\neq 6\wedge \{1,2,3,4,5\}\setminus\{i,j\}=\{a,b,c\}\\
\end{cases}
\end{dmath*}
and
\begin{dmath*}
w^{(1)}_1(f_i^*\wedge f^*_j \wedge f^*_k \wedge f_l^*)=\begin{cases}
0\quad\quad\operatorname{if}\quad l\neq 6\\
x_{ab}\quad\operatorname{if}\quad l=6,\{i,j,k,a,b,6\}=\{1,2,3,4,5,6\}.\\
\end{cases}
\end{dmath*}
In the second degree we have
\begin{dmath*}w^{(1)}_2(f_i^*)=
\begin{cases}
    \hat{e_i}\quad\textnormal{if}\quad 1\leq i\leq 5\\
    0\quad\textnormal{if}\quad i=6\\
\end{cases}
\end{dmath*}
and
\begin{dmath*}
w^{(1)}_2(\hat{f_i^*})=\begin{cases}
0\quad\quad\operatorname{if}\quad i\neq 6\\
1\quad\quad\operatorname{if}\quad i=6\\
\end{cases}
\end{dmath*}

Note that the computed top structure maps $w^{(1)}_2$ and $w^{(1)}_2$ are the differentials of a split exact complex. Combining this with results from the previous section, where we have shown that the top structure maps of a split exact complex are the differentials of the free resolution of a hyperplane section of pfaffians, we see that the duality between those two resolutions holds for the grading induced by $\alpha_2$ as well as for those induced by $\alpha_1$.

\section{Generic models for Gorenstein ideals of codimension four with six generators}\label{section:structure_theorems}

Now we are ready to prove the main theorem of the paper: any codimension four Gorenstein ideal $I\subset R$ with six generators is a hyperplane section of a codimension three Gorenstein ideal. If $I$ is generically complete intersection and $I/I^2$ is Cohen-Macaulay, it was proven in \cite{herzog-miller}. In \cite{vv} the second assumption was removed. In this section we show that the conjecture is true in general. Accordingly, we fix $n=6$ throughout this section.

\subsection{Defect variables and the Lie algebra $E_6$}
Before we discuss the proof of the main theorem, we revisit the defect variables introduced in \S\ref{sec:split,alpha1} and \S\ref{sec:split,alpha2}. These variables were used to parametrize the non-uniqueness of higher structure maps. In fact, this non-uniqueness can be recast in terms of subalgebras of $E_6$ as follows.

From Example~\ref{ex:1-grading}, we see that the negative part $\mf{n}_1^-\subset E_6$ in the $\alpha_1$-grading is the half-spinor representation $V(\omega_2,D_{n-1})$. The parametrization of higher structure maps is obtained by acting on the critical representations using the unipotent subgroup $N_1^-$ corresponding to the subalgebra $\mf{n}_1^*$. Indeed, the variables $b_H$ from \S\ref{sec:split,alpha1} should be viewed as elements of $(\mf{n}_1^-)^*$, as they give coordinates on $\mf{n}_1^-$ and consequently on $N_1^-$. For the sake of being as concrete and explicit as possible, we have opted to not use this formulation in the previous sections; the reader may consult \cite{Guerrieri-Ni-Weyman} for an explanation of how these ideas are related in the setting of length three resolutions.

Analogously, from Example~\ref{ex:2-grading}, the negative part $\mf{n}_2^-\subset E_6$ in the $\alpha_2$-grading is $\bigwedge^3 F_1^* \otimes F_4 \oplus \bigwedge^6 F_1^* \otimes S_2 F_4$. The variables $b_{i,j}^k$ and $c_1^1$ from \S\ref{sec:split,alpha2} should be viewed as elements of $(\mf{n}_2^-)^*$, giving coordinates on the first and second summands of $\mf{n}_2^-$ respectively. The parametrization of this other family of higher structure maps is given by the action of the unipotent subgroup $N_2^-$ corresponding to $\mf{n}_2^-$.

This perspective allows us to establish point (1) of \S\ref{subsec:motivation-3}. From \S\ref{sec:split,alpha2}, we see that by setting all defect variables to 0, the matrix $w^{(1)}$ has the form $\begin{bmatrix} 1 & 0 & \cdots & 0\end{bmatrix}$, which evidently satisfies the Pl\"ucker relations. Indeed, it is the Borel-fixed point in $G/P_1^+ \subset \mb{P}(V(\omega_1))$. Since $\operatorname{GL}(F)$, $\operatorname{SO}(G)$, and $N_2^-$ all act on this homogeneous space, it follows that an arbitrary choice of $w^{(1)}$ for an arbitrary split $\mb{F}$ yields a point of $G/P_1^+$.

If $\mb{F}$ is not split, the structure map $w^{(1)}$ still satisfies the Pl\"ucker relations because $\mb{F}$ becomes split after localization. To show that we still have a well-defined point of $\mb{P}(V(\omega_1))$, we need to verify that $w^{(1)}$ is surjective, which we turn to next.

\subsection{Surjectivity of $w^{(1)}$}

Let $V:=V(\omega_1,E_6)$ and consider the bi-grading on $V\otimes V$ induced by $(\alpha_1,\alpha_2)$. With respect to this bi-grading we have the decomposition:
$$
\Big(G\oplus V(\omega_2,D_5)\oplus \mathbb{C}\Big)\boxtimes\Big(F\oplus\bigwedge^4 F\oplus\bigwedge^6F\otimes F\Big)
$$

\begin{lem}\label{lem:square-matrix}
There exists an invertible matrix $M\in V \otimes V$ such that its restriction to $(k,0)$-part, resp. $(0,k)$-part, i.e. to

$$
\Big(G\oplus V(\omega_2,D_5)\oplus \mathbb{C}\Big)\boxtimes F,\quad
\textnormal{resp.}\quad \mathbb{C}\boxtimes\Big(F\oplus\bigwedge^4 F\oplus\bigwedge^6F\otimes F\Big) 
$$
is given by the higher structure maps in the $\alpha_1$-, resp. $\alpha_2$-grading.
\end{lem}

\begin{proof}
The structure of the proof is identical to the one detailed in \S4.1-4.3 of \cite{Guerrieri-Ni-Weyman}, so we will be brief.

First, computations in sections \ref{sec:split,alpha1} and \ref{sec:split,alpha2} show that if we specialize all defect variables to $0$, higher structure maps of $\mathbb{F}^{split}$ combine into a matrix
\[
M^{split}:=\kbordermatrix{
    & \mathbb{C} & u_1 & \cdots & u_5 & V(\omega_2,D_5)_{>1} & G \\
    f_1 & 0 & 1 & \cdots & 0 & \textbf{0}_{1,11} & \textbf{0}_{1,10}\\
    \cdots & \cdots & \cdots & \cdots & \cdots & \cdots &\cdots \\
    f_5 & 0 & 0 & \cdots & 1 & \textbf{0}_{1,11} & \textbf{0}_{1,10}\\
    f_6 & 1 & 0 & \cdots & 0 & \textbf{0}_{1,11} & \textbf{0}_{1,10}\\
    \bigwedge^4 F & \textbf{0}_{15,1} & \textbf{*}_{1,1} & \cdots & \textbf{*}_{1,1} & \textbf{*}_{15,11} & \textbf{*}_{15,10}\\
    \bigwedge^6 F\otimes F & \textbf{0}_{6,1} & \textbf{*}_{1,1} & \cdots & \textbf{*}_{1,1} & \textbf{*}_{6,11} & \textbf{*}_{6,10}\\
  }
\]
Here $V(\omega_2,D_5)_{>1}\subset V(\omega_2,D_5)$ denotes the subspace spanned by $u_J$ with $\# J>1$, $\textbf{0}_{a,b}$ is an $a\times b$ matrix of zeros and $\textbf{*}_{a,b}$ are yet unspecified $a\times b$ matrices. If we put $\textbf{*}_{1,1}=0$ everywhere and choose other $\textbf{*}_{a,b}$ to get $\operatorname{Id}_{21,21}$ in the bottom right corner, then the matrix $M^{split}$ is invertible, showing that the lemma holds for $\mb{F}^{split}$. Up to scale, $M^{split}$ is the unique  $E_6$-equivariant map $V \otimes V \to \mb{C}$. Alternatively, by identifying $V \cong V^*$ as $E_6$-representations, the matrix $M^{split}$ can be interpreted as the identity matrix.

Now suppose we have an arbitrary $\mb{F}$ resolving $R/I$ where $I$ is a codimension four Gorenstein ideal, with fixed choices of higher structure maps corresponding to both gradings. For any nonzerodivisor $h\in I$, the localization $\mb{F}\otimes R_h$ is split exact. Applying appropriate row and column operations (coming from the groups $\operatorname{SO}(G)$, $\operatorname{GL}(F)$, $N_1^-$, and $N_2^-$), we can transform $M^{split}$ into a matrix $M$ with the desired two restrictions.

A priori, the rest of this matrix has entries in $R_h$ rather than $R$. However, one can show that the output matrix is independent of the choice of operations used (or equivalently, the choice of splitting for $\mb{F}\otimes R_h$). As such, the entries of the matrix must live in every localization $R_h$, and thus in $R$ by algebraic Hartogs because the ideal $I$ has grade $4 \geq 2$. The same argument applies to the determinant of the matrix, so $M$ is invertible over $R$.
\end{proof}

Recall that the theorem of Buchsbaum and Eisenbud states that a codimension three Gorenstein ideal with $n$ generators is generated by the $(n-1)\times(n-1)$ pfaffians of a $n\times n$ skew-symmetric matrix (\cite[Theorem 2.1]{Buchsbaum-Eisenbud_codim-3}). As such, our goal is to prove that the hyperplane section of pfaffians, considered in sections \ref{sec:pfaffians,alpha1} and \ref{sec:pfaffians,alpha2}, is the generic example of a codimension four Gorenstein ideal with six generators.

\begin{thm}\label{th:generic-gorenstein}
Let $R$ be a complete regular local ring in which two is a unit or a graded polynomial ring over quadratically closed field of characteristic $\neq 2$. Let $I\subset R$ be a Gorenstein ideal of codimension four, minimally generated by six elements.
    
Then $I=(J,y)$, where $J$ is the ideal generated by submaximal pfaffians of some skew-symmetric $5\times 5$ matrix and $[y]$ is a regular element of $R/J$.    
\end{thm}

\begin{proof}

The method of the proof was already hinted at in \S\ref{subsec:motivation-3}. Let $G:=G(E_6)$ be the algebraic group with Lie algebra of type $E_6$ and let $P\subset G$ be the maximal parabolic subgroup corresponding to the vertex $\alpha_1$. The higher structure maps $w^{(1)}=w^{(1)}_0\oplus w^{(1)}_0\oplus w^{(1)}_2$ in the $\alpha_2$-grading live in the $(0,k)$-part of $V\boxtimes V$ and define an element of $R\boxtimes V(\omega_1,E_6)$. There is a standard embedding $G/P\hookrightarrow \mathbb{P}\left(V(\omega_1,E_6)\right)$ whose image is defined by Pl\"ucker relations. These relations are satisfied by the entries of $w^{(1)}$, so $f$ defines a map for $\operatorname{Spec} R$ to the affine cone over $G/P$. By the previous lemma $w^{(1)}$ is non-zero modulo the maximal ideal, hence $f$ gives a well-defined map $f:\operatorname{Spec}R\rightarrow G/P$. The lowest graded component $w^{(1)}_0$ of $w^{(1)}$ is given by the first differential, i.e. by the generators of the ideal $I$. By Example \ref{ex:Schubert-X^w} we conclude that under the map $f$ we have $f^{-1}X^w=\operatorname{Spec}R/I$, where $w=s_2s_4s_3s_1$ and $X^w$ is the appropriate opposite Schubert variety. We conclude that $R/I$ is a specialization of one of the local rings $\mathcal{O}_{\mathcal{N}_\sigma^w, \sigma v}$, where $\sigma \in W$ is a minimal length representative of its double coset in $W_{2} \backslash W / W_{1}$. But for the typ $E_6$ the double coset space has only three elements corresponding to the unit ideal, an almost complete intersection or a hyperplane section as in the statement of the theorem (see \cite{ftw}).
\end{proof}

\section{Conjectural models for Gorenstein ideals with seven and eight generators}\label{section:conjecture-end}

In this section we state a conjectural structure theorem for Gorenstein ideals of codimension four which are minimally generated by seven or eight elements. This is similar to Theorem \ref{th:generic-gorenstein} since it also predicts that every such ideal is a specialization of a Schubert variety. The precise computation of higher structure maps for the formats $E_7$ and $E_8$ should allow a proof of the conjecture completely analogous as our proof of Theorem \ref{th:generic-gorenstein}.

For the format $E_6$ there is only one generic model (only one Herzog class): a hyperplane section of a codimension three ideal. For ideals with seven generators an analysis of possible Schubert varieties shows two possibilities. We call them the \emph{Kustin-Miller model} (see \cite{KM1,KM2}) and the  \emph{generic model} (see \cite{examples}). 
This name 'generic' comes from the fact that, if one allows also non-local ring homomorphisms, it is possible to obtain the Kustin-Miller model as a specialization of the generic one.
Before stating the conjecture, we describe these two generic models for ideals with seven generators and show how to construct them using linkage.

We start with the Kustin-Miller model. Consider the following matrices whose entries are generic variables of degree 1:
$$
X=\begin{bmatrix}
    x_1\\
    x_2\\
    x_3\\
    x_4
\end{bmatrix},\quad
Y=\begin{bmatrix}
    y_{1,1} & y_{1,2} & y_{1,3} & y_{1,4}\\
    y_{2,1} & y_{2,2} & y_{2,3} & y_{2,4}\\
    y_{3,1} & y_{3,2} & y_{3,3} & y_{3,4}
\end{bmatrix},\quad
Z=YX=\begin{bmatrix}
    z_1\\
    z_2\\
    z_3
\end{bmatrix}=
\begin{bmatrix}
    \sum_{i=1}^4y_{1,i}x_i\\
    \sum_{i=1}^4y_{2,i}x_i\\
    \sum_{i=1}^4y_{3,i}x_i\\
\end{bmatrix}
$$
Let $v$ be a variable of degree $2$ and let $Y_i$, $i=1,2,3,4$, be the minor of $Y$ obtained by removing the $i$th column. Then, the Kustin-Miller model is the ideal $I_{KM}$ generated by the following seven elements:
$$
I_{KM}=\langle z_1,\ z_2,\ z_3,\ x_1v+Y_1,\ x_2v-Y_2,\ x_3v+Y_3,\ x_4v-Y_4\rangle
$$

Recall that two ideals $I,J\subset R$ of codimension $c$ are \emph{linked} if there exists a regular sequence $\underline{\alpha}=\alpha_1,\cdots,\alpha_c$ such that $I=(\underline{\alpha}):J$ and $J=(\underline{\alpha}):I$ (see \cite{linkage}). To obtain the ideal $I_{KM}$ by linkage, one proceeds as follows. We start with an almost complete intersection $I_1$ of codimension three. More precisely, let
$$
A=\begin{bmatrix}
    a_1\\
    a_2\\
    a_3
\end{bmatrix},\quad
B=\begin{bmatrix}
    b_{1,1} & b_{1,2} & b_{1,3}\\
    b_{2,1} & b_{2,2} & b_{2,3}\\
    b_{3,1} & b_{3,2} & b_{3,3}
\end{bmatrix},\quad
C=BA=\begin{bmatrix}
    c_1\\
    c_2\\
    c_3
\end{bmatrix}=
\begin{bmatrix}
    \sum_{i=1}^3b_{1,i}a_i\\
    \sum_{i=1}^3b_{2,i}a_i\\
    \sum_{i=1}^3b_{3,i}a_i\\
\end{bmatrix}
$$
Then $I_1=\langle c_1,c_2,c_3,\det B\rangle$ is a codimension three ideal minimally generated by four elements. Let $d$ be a variable and let $I_2=\langle I_1,d\rangle$ be a hyperplane section of $I_1$. It is an almost complete intersection of codimension four. We now introduce new variables $e_1,e_2,e_3,e_4$ of degrees $1,1,1,2$. Then $I_{KM}$ is linked to $I_2$ via the regular sequence
$\underline{\alpha}=c_1+e_1z,c_2+e_2z,c_3+e_3z,\det B+e_4z$, i.e. $I_{KM}=(\underline{\alpha}):I_2$.

The generic model $I_{NW}$ can be obtained from the Kustin-Miller model $I_{KM}$ via a double link. The first link is to an ideal $J=(\underline{\beta}):I_{KM}$, where $\underline{\beta}$ is the regular sequence 
$$
\beta_j:=x_jv+Y_j+\sum_{i=1}^3f_i^jz_i,\quad j=1,\cdots,4
$$
Here $f^j_i$ are new variables of degree $1$. The ideal $J$ is an almost complete intersection generated by $\beta_i$, $i=1,\cdots,4$, and an additional generator $\theta$ of degree $4$. Let $g$ be a new variable of degree 1. Then the generic model is the linked ideal $(\underline{\gamma}):J$, where $\underline{\gamma}$ is the regular sequence $\underline{\gamma}=\beta_1,\beta_2,\beta_3,g\cdot\beta_4+\theta$.

We expect that for Gorenstein ideals of codimension four with seven generators one can construct a family of higher structure maps living inside critical representations of $E_7$. The critical representations in the $\alpha_1$-grading decompose as follows:
$$W_1=F\otimes (G\oplus V(\omega_6, D_6)\oplus G),$$
$$W_7=F^*\otimes (\mathbb{C}\oplus V(\omega_2, D_6)\oplus [{\mathfrak{so}}(12)\oplus\mathbb{C}]\oplus V(\omega_2, D_6)\oplus \mathbb{C}),$$
$$W_0=  V(\omega_6, D_6)\oplus [V(\omega_1, D_6)\oplus V(\omega_3, D_6)]\oplus [V(\omega_1+\omega_6, D_6)\oplus V(\omega_2, D_6)]\oplus $$
$$\oplus  [V(\omega_1, D_6)\oplus V(\omega_3, D_6)]\oplus V(\omega_6, D_6),$$
while for the $\alpha_2$-grading the decompositions of the relevant representations are
$$
W_1=[F_1 \oplus\bigwedge^4 F_1 \boxtimes F_4^* \oplus \bigwedge^6F_1 \otimes F_1 \boxtimes S_2 F_4^* \oplus \bigwedge^7 F_1 \otimes \bigwedge^3 F_1 \boxtimes S_3 F_4^*] \boxtimes \mathbb{C}.
$$
$$
W_7=[F_1^* \oplus \bigwedge^2 F_1 \oplus \bigwedge^5 F_1 \oplus \bigwedge^7 F_1 \otimes F_1 ]\boxtimes F_2^*.
$$
The higher structure maps in the lower components are computed the same way as for the $E_6$ format.

We expect that computation of all higher structure maps will produce an invertible matrix as in Lemma \ref{lem:square-matrix}. Then an argument as in the proof of Theorem \ref{th:generic-gorenstein} will show that every Gorenstein ideal of codimension four minimally generated by seven elements is a specialization of an ideal of an opposite Schubert variety inside an open cell of the homogeneous space $E_7/P_{\alpha_1}$. There are only two such ideals which are Gorenstein of codimension four: one is given by the Kustin-Miller model, and the other by the generic model. Thus we state the following conjecture:

\begin{conjecture}
Let $R$ be a complete regular local ring in which two is a unit or a graded polynomial ring over quadratically closed field of characteristic $\neq 2$. Let $I\subset R$ be a Gorenstein ideal of codimension four, minimally generated by seven elements.
    
Then $I$ is a specialization of the Kustin-Miller model or of the generic model.
\end{conjecture}

The last $n$ for which the Lie algebra $E_n$ is finite-dimensional is $n=8$. Thus it is natural to conjecture the existence of the higher structure maps, and hence appropriate structure theorems, also for Gorenstein ideals of codimension four with eight generators. More precisely, we conjecture that for every such ideal $I\subset R$ there exists a ring homomorphism $R_{cell}\rightarrow R$, where $R_{cell}$ is the coordinate ring of some Schubert cell inside the homogeneous space $E_8/P_{\alpha_1}$ and the ideal of the Schubert variety $X^{s_2s_4s_3s_1}$ specializes to $I$. We do not go into details of these generic models but for the sake of completeness we include all possible Betti tables.

\begin{small}
\begin{tabular}{c|c c c c c}
	&0&1&2&3&4\\
	\hline
	0&1&&&&\\
	1&&1&&&\\
	2&&7&14&7&\\
	3&&&&1&\\
	4&&&&&1\\
\end{tabular}\quad\quad\quad
\begin{tabular}{c|c c c c c}
	&0&1&2&3&4\\
	\hline
	0&1&&&&\\
	1&&1&&&\\
	2&&3&&&\\
	3&&4&14&4&\\
	4&&&&3&\\
	5&&&&1&\\
	6&&&&&1\\
\end{tabular}\quad\quad\quad
\begin{tabular}{c|c c c c c}
	&0&1&2&3&4\\
	\hline
	0&1&&&&\\
	1&&&&&\\
	2&&2&&&\\
	3&&4&&&\\
	4&&2&14&2&\\
	5&&&&4&\\
	6&&&&2&\\
	7&&&&&\\
	8&&&&&1\\
\end{tabular}\quad\quad\quad

\medskip

\begin{tabular}{c|c c c c c}
	&0&1&2&3&4\\
	\hline
	0&1&&&&\\
	1&&&&&\\
	2&&&&&\\
	3&&4&&&\\
	4&&3&&&\\
	5&&1&14&1&\\
	6&&&&3&\\
	7&&&&4&\\
	8&&&&&\\
	9&&&&&\\
	10&&&&&1\\
\end{tabular}\quad\quad\quad
\begin{tabular}{c|c c c c c}
	&0&1&2&3&4\\
	\hline
	0&1&&&&\\
	1&&&&&\\
	2&&1&&&\\
	3&&1&&&\\
	4&&6&&&\\
	5&&&14&&\\
	6&&&&6&\\
	7&&&&1&\\
	8&&&&1&\\
	9&&&&&\\
	10&&&&&1\\
\end{tabular}\quad\quad\quad
\begin{tabular}{c|c c c c c}
	&0&1&2&3&4\\
	\hline
	0&1&&&&\\
	1&&&&&\\
	2&&&&&\\
	3&&&&&\\
	4&&7&&&\\
	5&&&&&\\
	6&&1&14&1&\\
	7&&&&&\\
	8&&&&7&\\
	9&&&&&\\
	10&&&&&\\
	11&&&&&\\
	12&&&&&1\\
\end{tabular}\quad\quad\quad

\medskip

\begin{tabular}{c|c c c c c}
	&0&1&2&3&4\\
	\hline
	0&1&&&&\\
	1&&&&&\\
	2&&&&&\\
	3&&1&&&\\
	4&&4&&&\\
	5&&3&&&\\
	6&&&14&&\\
	7&&&&3&\\
	8&&&&4&\\
	9&&&&1&\\
	10&&&&&\\
	11&&&&&\\
	12&&&&&1\\
\end{tabular}\quad\quad\quad
\begin{tabular}{c|c c c c c}
	&0&1&2&3&4\\
	\hline
	0&1&&&&\\
	1&&&&&\\
	2&&&&&\\
	3&&&&&\\
	4&&2&&&\\
	5&&5&&&\\
	6&&1&&&\\
	7&&&14&&\\
	8&&&&1&\\
	9&&&&2&\\
	10&&&&5&\\
	11&&&&&\\
	12&&&&&\\
	13&&&&&\\
	14&&&&&1\\
\end{tabular}\quad\quad\quad
\begin{tabular}{c|c c c c c}
	&0&1&2&3&4\\
	\hline
	0&1&&&&\\
	1&&&&&\\
	2&&&&&\\
	3&&&&&\\
	4&&&&&\\
	5&&4&&&\\
	6&&4&1&&\\
	7&&1&&&\\
	8&&&14&&\\
	9&&&&1&\\
	10&&&1&4&\\
	11&&&&4&\\
	12&&&&&\\
	13&&&&&\\
	14&&&&&\\
	15&&&&&\\
	16&&&&&1\\
\end{tabular}\quad\quad\quad

\medskip

\begin{center}
\begin{tabular}{c|c c c c c}
	&0&1&2&3&4\\
	\hline
	0&1&&&&\\
	1&&&&&\\
	2&&&&&\\
	3&&&&&\\
	4&&&&&\\
	5&&&&&\\
	6&&7&&&\\
	7&&1&&&\\
	8&&&&&\\
	9&&&14&&\\
	10&&&&&\\
	11&&&&1&\\
	12&&&&7&\\
	13&&&&&\\
	14&&&&&\\
	15&&&&&\\
	16&&&&&\\
	17&&&&&\\
	18&&&&&1\\
\end{tabular}\quad\quad\quad\quad\quad\quad
\end{center}
\end{small}

\bigskip

\noindent The first of these Betti tables is the one of an hyperplane section of an Gorenstein ideal of codimension 3 and 7 generators. The last Betti table is the one of the generic example constructed in \cite{examples}. As it happens for the $E_7$ models, also in this case all the ideals corresponding to opposite Schubert varieties associated to these Betti tables can be obtained iteratively by linkage, using
a sequence of double links.

\section*{Acknowledgements}

This material is partially based upon work supported by the National Science Foundation under Grant No. DMS-1928930 and by the Alfred P. Sloan Foundation under grant G-2021-16778, while the authors were in residence at the Simons Laufer Mathematical Sciences Institute (formerly MSRI) in Berkeley, California, during the Spring 2024 semester. 
 
T. Chmiel, L. Guerrieri and J. Weyman are supported by the grants MAESTRO NCN-UMO-2019/34/A/ST1/00263 - Research in Commutative Algebra and Representation Theory, NAWA POWROTY- PPN/PPO/2018/1/00013/U/00001 - Applications of Lie algebras to Commutative Algebra, and OPUS grant National Science Centre, Poland grant UMO-2018/29/BST1/01290. 
 
L. Guerrieri is also supported by the Miniatura grant 2023/07/X/ST1/01329 from NCN (Narodowe Centrum Nauki), which funded his visit to SLMath in April 2024.
 
The authors would like to thank Ela Celikbas, Lars Christensen, David Eisenbud, Sara Angela Filippini, Craig Huneke, Witold Kra\'skiewicz, Andrew Kustin, Jai Laxmi, Claudia Polini, Steven Sam, Jacinta Torres, Bernd Ulrich, and Oana Veliche for interesting discussions
 pertaining to this paper and related topics.


\begin{thebibliography}{100}

\bibitem{bruns} W. Bruns. \emph{The existence of generic free resolutions and related objects.} Math. Scand. 55 (1984), 33-46.

\bibitem{Buchsbaum-Eisenbud_codim-3}D. Buchsbaum, D. Eisenbud. \emph{Algebra Structures for Finite Free Resolutions, and Some Structure Theorems for Ideals of Codimension 3}. American Journal of Mathematics, Vol. 99(3), 447–485 (1977).

\bibitem{BE74} 
D. Buchsbaum, D. Eisenbud, 
\it Some structure theorems for finite free resolutions, \rm
Advances in Math. 1 (1974), 84-139

\bibitem{celikbas-laxmi-weyman}E. Celikbas, J. Laxmi, J. Weyman. \emph{Spinor structures on free resolutions of codimension four Gorenstein ideals}. Osaka J. Math., Vol. 60(4), 903-931 (2023).

\bibitem{Gue-SAF} S. A. Filippini, L. Guerrieri, \it
Mapping resolutions of length three II - Module formats,  \rm Linear Algebra and its Applications 704, 1-34 (2025)

\bibitem{ftw}S. A. Filippini, J. Torres, J. Weyman. \emph{Minuscule Schubert varieties of exceptional
type}. Journal of Algebra (619), 1–25 (2023).

\bibitem{Guerrieri-Ni-Weyman} L. Guerrieri, X. Ni, J. Weyman. \emph{An ADE correspondence for grade three perfect ideals}. arXiv:2407.02380.

\bibitem{GNW1} L. Guerrieri, X. Ni, J. Weyman, \it Higher structure maps for free resolutions of length 3 and linkage, \rm arXiv:2208.05934 

\bibitem{GNW3} L. Guerrieri, X. Ni, J. Weyman, \it The linkage class of a grade three complete intersection, \rm arXiv:2412.00399.

\bibitem{Gue-Wey} L. Guerrieri, J. Weyman, \it
Mapping resolutions of length three I, \rm J.
Pure Appl. Algebra 227.4 (2023), No. 107248, 22.

\bibitem{Herzog} J. Herzog, \it Deformationen von Cohen-Macaulay algebren, \rm J. Reine Angew. Math. 318 (1980) 83–105

\bibitem{herzog-miller} J. Herzog, M. Miller. \emph{Gorenstein ideals of deviation two.} Communications in Algebra (13), 1977-1990 (1985).

\bibitem{H75} M. Hochster. \emph{Topics in the Homological Theory of Modules over Commutative Rings}, CBMS Regional Conference Series in  Mathematics vol. 24, 1975.

\bibitem{huneke} C. Huneke. \emph{The arithmetic perfection of Buchsbaum–Eisenbud varieties and generic modules of projective dimension two.} Trans. Amer. Math. Soc. 265 (1981) 211–233.

\bibitem{kunz}E. Kunz. \emph{Almost complete intersections are not Gorenstein rings.} J. Algebra 28, 111-115 (1974).

\bibitem{DGA-algebra} A.R. Kustin, M. Miller. \emph{Algebra structure on minimal resolution of Gorenstein rings of embedding codimension four.} Math. Z. 173, 171–184 (1980).

\bibitem{KM1}  A.R. Kustin, M. Miller. \emph{Constructing big Gorenstein ideals from small ones.} J. Algebra 85, 303–322 (1983).

\bibitem{KM2} A.R. Kustin, M. Miller. \emph{Deformation and linkage of Gorenstein algebras.} Trans. Amer. Math. Soc.
284, 501–534 (1984).

\bibitem{Lee-Weyman} K.-H. Lee, J. Weyman. \emph{Some branching formulas for Kac-Moody Lie algebras.} Commun. Korean Math. Soc. 34(4), 1079-1098 (2019).

\bibitem{examples} X. Ni, J. Weyman. \emph{Free resolutions constructed from bigradings on Lie algebras.} arXiv:2304.01381.

\bibitem{palmer}S. Palmer. \emph{Algebra structures on resolutions of rings defined by grade four almost complete intersections}. Journal of Algebra (159), 1-46 (1993).

\bibitem{linkage} C. Peskine, L. Szpiro. \emph{Liaison des vari\'et\'es alg\'ebriques.} Invent. Math. 26, 271–302 (1974).

\bibitem{reid}M. Reid. \emph{Gorenstein in codimension 4: the general structure theory.} Adv. Stud. Pure Math., 201-227 (2015).

\bibitem{vv}  W. V. Vasconcelos, R. Villarreal. \emph{On Gorenstein ideals of codimension four.} Proceedings Of The American Mathematical Society, Vol. 98(2), 205-210 (1986).

\bibitem{Weyman89} J. Weyman, \emph{On the structure of resolutions of length 3.} Journal of Algebra 126 No. 1 (1989), 1-33.

\bibitem{W18}  J. Weyman, \emph{Free resolutions and root systems.} Annales de l'Institute Fourier 68 (3) (2018) 1241-1296

\bibitem{weyman-gorenstein}J. Weyman. \emph{Higher structure theorems for codimension four Gorenstein ideals}. \mbox{\url{https://alg.matinf.uj.edu.pl/preprints.html}.}

\end{thebibliography}
\end{document}